\documentclass{amsart}
\newcommand{\scal}{{\mathcal S}}
\usepackage{epsfig}
\usepackage{bbm}
\usepackage{amscd}
\usepackage{amsmath, amssymb, amsthm}
\usepackage{graphics}
\usepackage{color}
\usepackage{dsfont}
\newtheorem*{main-theorem}{Main Theorem}
\newtheorem{proposition}{Proposition}[section]
\newtheorem{theorem}{Theorem}
\newtheorem*{old-thm}{Theorem}
\newtheorem{lemma}[proposition]{Lemma}
\newtheorem{corollary}[proposition]{Corollary}

\theoremstyle{definition}

\newtheorem*{remark}{Remark}

\numberwithin{equation}{section}

\def\C{{\mathbb C}}

\def\11{\mathds{1}}

\def\ZZ{{\mathbb Z}}
\def\reals{{\mathbb R}}

\def\Ci{{\mathcal C}^\infty}
\def\Re{\,\mathrm{Re}\,}
\def\Im{\,\mathrm{Im}\,}

\def\WF{\mathrm{WF}_h\,}

\def\O{{\mathcal O}}

\def\s{{\mathcal S}}

\def\nbhd{\mathrm{neigh}\,}

\def\phi{\varphi}

\def\be{\begin{eqnarray*}}
\def\ee{\end{eqnarray*}}
\def\ben{\begin{eqnarray}}
\def\een{\end{eqnarray}}

\def\lll{\left\langle}
\def\rrr{\right\rangle}

\def\L2R{L_{\text{Rest}}^2}

\def\tchi{\tilde{\chi}}

\def\TT{\mathbb{T}}

\begin{document}
\title[Quantum Flux]{Quantum Flux and Quantum Ergodicity for Cross Sections}
\author{Hans Christianson}

\author{John Toth}

\address{}
\email{}

% \keywords{quantum flux, quantum ergodicity}
\begin{abstract}
For sequences of quantum ergodic eigenfunctions, we
define the quantum flux norm associated to a codimension $1$
submanifold $\Sigma$ of a non-degenerate energy surface.   We prove restrictions
of eigenfunctions to $\Sigma$, realized using the quantum flux norm,
are quantum ergodic.  We compare this result to known results from
\cite{CTZ} in the case of 
Euclidean domains and hyperfurfaces.
As a further application, we consider complexified analytic eigenfunctions and prove a second microlocal analogue of \cite{CTZ} in that context.
\end{abstract}
\maketitle

\begin{center}
In memory of Steve Zelditch; mentor, collaborator, and friend.
\end{center}
%\setcounter{section}{-1}

%%%%%%%%%%%%%%%%%%%%%%%%%%%%%%%%%%%%%%%%%%%%%%%%%%%%%%%%%%%%%%%%%%%%%%%%%%%%%%%%%%%%%%%%%%%%%%%%%%%%%%%%%%%%%%%%%%%%%%%%%%%%%%%%%%%%%%%%%%%%%%%%%%%%%%%%%%%%%%%%%%%%%%%%%%%%%%%%%%%%%%%%%%%%%%%%%%%%%%%%%%%%%%%%%%%%%%%%%%%%%%%%%%%%%%%%%%%%%%%%%%%%%%%%%%%%%%%%%%%%%%%%%%%%%%%%%%%%%%%%%%%%%%%%%%%%%%%%%%%%%%%%%%%%%%%%%%%%%%%%%%%%%%%%%%%%%%%%%%%%%%%%%%%%%%%%%%%%%%%%%%%%%%%%%%%%

\section{Introduction}
\label{introduction}

% $\clubsuit$ Note to self: fix Weyl's law and QE statements in section
% 3, and rewrite QF lemma as in notes.  $\clubsuit$

The purpose of this note is to prove a general and abstract version of the QER (quantum ergodic
restriction) theorem in \cite{CTZ} where a hypersurface $H \subset X$ of
a compact Riemannian manifold is replaced by a general hypersurface in
the energy shell of the principal symbol of a pseudodifferential operator. The QER theorem of \cite{CTZ} proves that the Cauchy data of eigenfunctions
along hypersurfaces $H \subset X$ is quantum ergodic along the hypersurface $H$.
 In the special case where $H = \partial X$, the boundary of
$X$, and where the eigenfunctions satisfy standard boundary conditions (DIrichlet, Neumann, Robin),
the QER property was proved earlier in \cite{HZ}.    The general QER theorem of
this article proves  a phase space analogue for  generalized Cauchy data of eigenfunctions along   cross sections
to the classical flow in an energy shell in $T^*X$.  
     Let $X$ be a
smooth, compact  manifold of dimension $n$ and assume for simplicity in
this work that $\partial X = \emptyset$.  Let $P(h) \in
\Psi_h^{k,0}(X)$
%\footnote{We refer to \cite{Zwo-book} for the standard
%  definitions in semiclassical analysis.}
be a self-adjoint semiclassical pseudodifferential operator with principal symbol
$p(x,\xi) \in \s^{k,0}(T^*X)$ of real principal type.  That is, $p$ is
real-valued and independent of $h$.  Assume $p$ satisfies $dp
\neq 0$ on $\{ p=0 \}$ and $p$ is elliptic outside a compact set, that
is, 
there exists $C>0$ such that $|p| \geq \lll \xi \rrr^k /C$ whenever
$|\xi| >C$.  
Let $\phi_t  = \exp(tH_p)$ be the Hamilton flow of $p$, and
let $\Sigma \subset \{p=0\}$ be a compact, codimension $1$ submanifold.    We
prove that for any cross section, quantum ergodicity of a sequence of
eigenfunctions 
implies the microlocal restrictions of
eigenfunctions to $\Sigma$ are quantum ergodic with respect to a
natural measure. We prove this result by means of
associating a {\it quantum flux norm} to a sequence of quantum ergodic
eigenfunctions $P(h) u_j(h) =
E_j(h) u_j(h)$ near $\Sigma$.  We will assume throughout that the eigenfunctions $u_j(h)$ are all real-valued.

To state the result we need further notation. We say the system 
\begin{equation} \label{EV} P(h) u_j = E_j(h) u_j \end{equation} is quantum ergodic 
if there exists a density one subsequence $\scal$, such that for every semi-classical pseudo-differential operator $Op_h(a)$ on $X$
$$\lim_{h_j \to 0, h_j \in \scal} \langle Op_h(a) u_j, u_j \rangle  = \int_{p^{-1}(0)} a_0 d\mu_p, $$
where $d\mu_p$ is Liouville measure on the energy surface $\{p = 0\}
$, and $a_0$ is the semiclassical principal symbol of $Op_h(a)$.  For background we refer
to  \cite{TZ2, CTZ,Zwo-book} or for the original articles on quantum ergodicity,  \cite{Sni, Zel1,CdV, HMR,ZZ}.  

In this note, we use the term {\it quantum ergodic sequence} to refer to the sequence $\scal$ of eigenfunctions satisfying the quantum ergodic limit.  
The following theorem
asserts that the {\it phase space} restrictions of these quantum ergodic eigenfunctions to a codimension
 one hypersurface $\Sigma \subset \{p=0\}$ are also quantum ergodic
with respect to a natural measure.

\begin{theorem}
\label{theorem-1}
Assume $\{ u_j \}$ is a 
quantum ergodic sequence of eigenfunctions of $P(h)$, $P(h) u_j =
E_j(h) u_j$ with $E_j(h) = o(1)$.  Then for any  $A \in \Psi_{h, \Sigma,
  \text{fi}}^{0,0}(X)$, as $h \to 0+$,  
\begin{equation}
\label{E:QF-1}
\lll A u_j, u_j \rrr_{QF} \rightarrow \frac{1}{(n-1)! \mu ( \{p=0\})} \int_\Sigma
\sigma_h(A)(H_p \beta)|_{\Sigma} (\omega |_{\Sigma})^{n-1}.
\end{equation}
Here $\lll \cdot, \cdot \rrr_{QF} = \lll \cdot, \cdot \rrr_{QF( \Sigma)}$ denotes the quantum flux norm
associated to $\Sigma$, $\mu$ is the Liouville measure on $\{p=0\}$, $\omega$ is the symplectic form on $T^*X$, and $\beta$ is a normalized defining function for $\Sigma$ compatible with the quantum flux norm.
\end{theorem}
Here the operator class  $\Psi_{h, \Sigma, \text{fi}}^{0,0}(X)$ defined
precisely below in \eqref{E:fi-def} is the space of operators which
commute microlocally with $P$ near $\Sigma$.  We lose nothing from
this definition; it is only to make the quantum flux norm independent
of the choice of cutoffs.  However we can dispense with this if we are
more careful in our choice of the order of limiting processes (see \S
\ref{SS:cor} below).

\begin{remark}
In particular, if the whole sequence of eigenfunctions $\{ u_j \}$ is quantum uniquely ergodic, then the
proof of Theorem \ref{theorem-1} shows the quantum flux norms are also
quantum uniquely ergodic.

\end{remark}

%%%%%%%%%%%%%%%%%%%%%%%%%%%%%%%%%%
%% compute LHS.  Flesh this out more...
%%%%%%%%%%%%%%%%%%%%%%%%%%%%%%%%%%
\begin{remark}
It is of course difficult to interpret the left hand side of
\eqref{E:QF-1} without an explicit formula.  Let us assume we are
working in a microlocal neighbourhood of a point $(x_0, \xi_0) \in
T^*X$ in which $\Sigma \subset \{p = 0 \} \subset T^*X$ is a symplectic submanifold; that is,
$\omega|_\Sigma$ is a non-degenerate 2-form.  Let $\nu$ be a unit
normal vector to $\Sigma$, and assume $\beta \in \Ci ( T^*X )$
satisfies $\nu = \nabla \beta$.  Then since $dp \neq 0$ on $\{ p = 0
\}$, we must have $\{ p, \beta \} \neq 0$, since otherwise either $dp
= 0$, $d\beta = 0$, or $d p$ is parallel to $d\beta$, none of which
can be true.  Then we can choose symplectic coordinates such that $p =
\xi_n$ and $\beta = x_n$ in a neighbourhood of $(x_0, \xi_0)$.  

Now for $A$ to be a microlocally flow invariant pseudodifferential
operator near $\Sigma$, that means that $[A, P] = 0$ microlocally, or
that $A$ is independent of $x_n$.  Hence $A = A(x', hD', hD_n )$.  The
dependence on $(x', hD')$ is tangential to $\Sigma$, while $hD_n$ is
normal.  We observe that this representation cannot possibly be
unique, since any symplectic transformation in $(x', \xi')$ preserves
this decomposition but alters the symbol $A$ according to Egorov's
theorem.  However, once a {\it polarization} is chosen, that is, a
choice of horizontal $x'$ directions and vertical $\xi'$ directions,
then the restrictions of eigenfunctions to $H = \{ x_n = 0 \} \subset
\{ (x', x_n) \}$ is well-defined via the FBI transform.  In this case,
the quantum flux norm has an explicit formula.

To make this precise, let $u \in \Ci$ have semiclassical wavefront set
sufficiently close to $(x_0, \xi_0)$.  This is the case for example if $u$ is an
eigenfunction and if the symbol for $A$ has microsupport near $(x_0,
\xi_0)$.  Then we can cut off $u$ to a slightly larger microlocal
neighbourhood and assume the microsupport properties of $u$.  
Then for $(x, \xi)$ near $(x_0, \xi_0)$ let 
\[
T_0 u(x, \xi) = c_n h^{-3n/4} \int u(y) e^{i \xi ( x-y)/h} e^{-(x-y)^2 /2h } dy
\]
be the (semiclassical) FBI transform of $u$.  We have
\[
T_0 : \s( \reals^n) \to \s( \reals^{2n} ),
\]
and $T_0^*: \s( \reals^{2n}) \to \s( \reals^{n} )$ is given by
\[
T_0^* v(y) = c_n h^{-3n/4} \int v(x, \xi )e^{-i \xi (x-y) /h}
e^{-(x-y)^2 / 2h} d x d \xi .
\]
It is then classical that for a suitable constant $c_n$,
\[
T_0^* T_0 = I: L^2(dx) \to L^2(dx).
\]
We let also $T_1$ be the FBI transform in the $(x', \xi')$ variables.  

Now let $\kappa: \nbhd ( x_0, \xi_0) \to \nbhd ( x_0, \xi_0)$ be a
local symplectomorphism taking $(p, \beta)$ to $(\xi_n, x_n )$.  Let
$F$ be the semiclassical quantization of $\kappa$.  Then
\[
u|_H(x') = T_1^* F T_0 u ( x' )
\]
is the microlocal restriction of $u$ to (the projection $H$) of
$\Sigma$ given this polarization, and
\[
\lll A u_j, u_j \rrr_{QF} = \lll A(x', hD', hD_n ) u|_H, u|_H
\rrr_{L^2(H)}.
\]
The FBI transform in the context of complexified eigenfunctions is discussed in depth in Section \ref{s:FBIDescription}.

\end{remark}

%%%%%%%%%%%%%%%%%%%%%%%%%%%%%%%%%%
%% compute LHS
%%%%%%%%%%%%%%%%%%%%%%%%%%%%%%%%%%

Now let $H$ be a submanifold in the base manifold $X$, and let $P =
-h^2 \Delta -1$ with principal symbol $p = | \xi |^2 -1$.  As we will see, if $\Sigma$ is the boundary of an open submanifold, the
quantum flux norm vanishes as $h \to 0$, so we cannot immediately draw
a conclusion about the quantum flux norm if $\Sigma = S^*_H X$, which
is the case studied in \cite{CTZ}.  However, by associating a
quantum flux norm to the incoming/outgoing directions in $S^*_H X$
(that is, the right and left hemispheres in the cosphere bundle
restricted to $H$), we are nevertheless able to recover the main
result in \cite{CTZ}.  This is summarized in the following
Corollary.

\begin{corollary}

\label{C:cor}
Suppose $X$ is a compact Riemannian manifold with or without boundary, $P = -h^2 \Delta
-1$ is the (Dirichlet) Laplacian on $X$, and 
$H \subset
X$ is a smooth, codimension $1$ relatively open submanifold.  If
$\{ u_j \}$ is an orthonormal sequence of quantum ergodic
eigenfunctions $Pu_j = 0$ and $a^w \in \Psi^0(H)$, then as $h \to 0+$,
\begin{align}
\label{E:rellich-thm}
& \lll hD_\nu a^w u_j, u_j \rrr_{\text{QF}(H+)} + \lll h D_\nu a^w u_j,u_j
\rrr_{\text{QF}(H-)} \\
& \quad \to \frac{4}{\mu(S^*
 X )} \int_{B^*H} a_0(x', \xi') (1 - | \xi' |^2)^{1/2} d \sigma, \notag
\end{align}
where $a_0$ is the principal symbol of $a^w$, 
$\mu$ is the Liouville measure on $S^* X$, and $d \sigma$
is the standard symplectic volume measure on $B^* H$.

\end{corollary}

These results on quantum ergodicity of Cauchy data are more general but less
precise than the results in \cite{TZ1,TZ2,DZ} on QER for Dirichlet data. These
articles find a condition on a hypersurface $H$ which is sufficient for the restrictions of
eigenfunctions to be quantum ergodic along the hypersurface. The condition is an asymmetry condition on the
hypersurface with respect to geodesics, which we do not make in the
present article.\\

Next
consider the special case where  $X$ is real-analytic and $X^{\mathbb C}$ denote the (maximal) Grauert tube complexification. By a theorem of Bruhat-Whitney, $X^{\mathbb C}$ can be naturally identified $B_{r}^*X$ with maximal $r>0$ under the complexified exponential map. By possibly rescaling $h$,   we assume in the following that $r>1.$ $X^{\mathbb C}$ inherits a natural Kahler structure with Kahler form $\omega$ that is the pull-back of the canonical symplectic form on $B_{r}^*X$ under the complexified exponential map $\kappa: B_r^*X \to X^{\mathbb C}$ where $\kappa(x,\xi) = \exp_{x} (i \xi).$ In this case, 
 $\rho(x,\xi)= \frac{1}{2} |\xi|_{x}^2$ is the associated Kahler potential function with $\partial \bar{\partial} \rho = \omega.$
Moreover, in the following, we let $-h^2 \Delta_{\bar{\partial}}$ denote the Kahler Laplacian corresponding to the natural Kahler structure on the Grauert tube $X^{\mathbb C}.$

Given a separarting hypersurface $ \Sigma  \subset X^{\mathbb C}$, one can define the notion of a {\em 2 microlocal quantum flux (2QF)}   as in (\ref{2QF}) to study the asymptotics of the restrictions to $\Sigma \subset X^{\C}$ of matrix elements of the analytic continuations, $u_h^{\C}$, of the eigenfunctions in the Grauert tube quite explicitly.  In the following, we let $(\beta', \beta): U_{\Sigma} \to {\mathbb R}^{2n}$ be local Fermi coordinates in neighbourhood, $U_{\Sigma},$ of $\Sigma$  (see subsection \ref{FERMI}) with $\Sigma = \{ \beta = 0 \}$ and unit exterior normal $\partial_{\beta}. $ We abuse notation somewhat  and continue to denote this 2-microlocal quantum flux by $\langle \cdot, \cdot \rangle_{2QF(\Sigma)}$ when the context is clear. First, we prove the following 2-microlocal analogue of  Corollary \ref{C:cor}.
\begin{proposition} \label{2microlocal}
Let $X$ be a compact $C^{\omega}$-manifold and $\Sigma \subset X^{\C}$ a compact, separating hypersurface in the Grauert tube with Kahler potential function $\rho(x,\xi):= \frac{1}{2} |\xi|^{2}_x.$ Assume that  $\Sigma \cap S^*X $  is a codimension-one submanifold of $S^*X$  and $\{ u_j \}$ is a 
quantum ergodic sequence of eigenfunctions of $P(h)$, $$P(h) u_j =
E_j(h) u_j$$ with $E_j(h) = o(1)$ and  let $u_j^{\C}$ be the holomorphic continuation to $X^{\C}.$ Then, there exists $q \in S^0$ such that for any  $a \in C^{\infty}_0(X^{\C}),$

 \begin{eqnarray}
\label{2micro}
\langle hD_{\nu}a e^{-\rho/h} u_j^{\C}, e^{-\rho/h} u_j^{\C} \rangle_{2QF(\Sigma^+)}  + \langle e^{-\rho/h} hD_{\nu}  a u_j^{\C},  e^{-\rho/h} u_j^{\C} \rangle_{2QF(\Sigma^-)} \nonumber \\
  \sim_{h \to 0^+} \int_{\Sigma \cap S^*X} a \, q \, d\mu_{\Sigma}
\end{eqnarray}
Here $\lll \cdot, \cdot \rrr_{2QF} = \lll \cdot, \cdot \rrr_{2QF( \Sigma)}$ denotes the $2$-microlocal quantum flux (see (\ref{2QF})
associated with $\Sigma$  and $d\mu_{\Sigma}$ is the measure on $\Sigma \cap S^*X$ induced by Liouville measure on $S^*X.$\end{proposition}

Finally, using the fact that $-h^2 \Delta_{\bar{\partial}} u_h^{\mathbb C} = 0$, we combine  Proposition \ref{2microlocal} with a suitably adapted Rellich-type commutator argument to prove the following Theorem.
\newpage

\begin{theorem} \label{2microlocalQER}
Under the same assumptions as in Proposition \ref{2microlocal} and letting $$-h^2 \Delta_{\Sigma}: C^{\infty}(\Sigma) \to C^{\infty}(\Sigma)$$ be the Laplacian on $\Sigma$ induced by  the Kahler Laplacian $-h^2 \Delta_{\bar{\partial}},$  it follows that there exists $q \in S^{0}$ such that for any $a \in C^{\infty}_{0}(\Sigma),$

$$   \big\langle   a \,  \big( \, h^2 \Delta_{\Sigma} +  2h \nabla \rho  + h \Delta \rho \big) \, e^{-\rho/h} u_h^{\mathbb C}, \, e^{-\rho/h} u_h^{\mathbb C} \big\rangle_{L^2(\Sigma)}  
 + \big\langle a\, h \partial_{\beta}  \, e^{-\rho/h} u_h^{\mathbb C}, \, h \partial_{\beta} \, e^{-\rho/h} u_h^{\mathbb C}  \big\rangle_{L^2(\Sigma)} \, $$
 $$ \sim_{h \to 0^+} e^{1/h} \int_{\Sigma \cap S^*X} a \, q \, d\mu_{\Sigma}.$$
  \end{theorem}

The formula for the symbol $q$ in somewhat cumbersome to write; it is equal to $q_1 + q_2$ where $q_1$  appears in Lemma \ref{sub} (see (\ref{q1term}) in the appendix) and $q_2$ appears in (\ref{q2term}).
One can  view Theorem \ref{2microlocalQER} as the natural 2-microlocal analogue of the QERCD result in \cite{CTZ} applied to the complexified eigenfunctions $u_h^{\mathbb C}$ in the Grauert tube $X^{\mathbb C}.$

\subsection*{Acknowledgements}

This article arose from a joint project with Steve Zelditch that began almost a decade ago. Unfortunately, due to his untimely passing, he did not have an opportunity to see the final version. Both authors wish to acknowledge  his deep insights and  crucial input.

The approach in this article was partly developed while H.C. was still
a graduate student at UC-Berkeley.  He would like to thank his advisor
M. Zworski for many helpful discussions and suggestions for this work.

A portion of the research of H.C. was partially supported by NSF grant \# DMS-0900524;
J.T. was partially supported by NSERC grant \# OGP0170280.

%%%%%%%%%%%%%%%%%%%%%%%%%%%%%%%%%%%%%%%%%%%%%%%%%%%%%%%%%%%%%%%%%%%%%%%%%%%%%%%%%%%%%%%%%%%%%%%%%%%%%%%%%%%%%%%%%%%%%%%%%%%%%%%%%%%%%%%%%%%%%%%%%%%%%%%%%%%%%%%%%%%%%%%%%%%%%%%%%%%%%%%%%%%%%%%%%%%%%%%%%%%%%%%%%%%%%%%%%%%%%%%%%%%%%%%%%%%%%%%%%%%%%%%%%%%%%%%%%%%%%%%%%%%%%%%%%%%%%%%%%%%%%%%%%%%%%%%%%%%%%%%%%%%%%%%%%%%%%%%%%%%%%%%%%%%%%%%%%%%%%%%%%%%%%%%%%%%%%%%%%%%%%%%%%%%%

\section{Preliminaries}
\label{preliminaries}

We first define the quantum flux norm for an arbitrary cross section
in $\Sigma$.  We remark that the quantum flux norm has previously been used  in
\cite{ISZ,SjZw-mono,Chr-QMNC} to relate eigenfunctions in a tubular
neighbourhood of a periodic geodesic to their microlocal restrictions
to the Poincar\'e section.  The current use is slightly different, so
we give some details.  
Choose an orientation for $\Sigma$ in $\{p=0\}$
and let $\nu$ be the unit normal field.  Let $\tchi \in \Ci(T^*X)$ be a
cutoff to $\Sigma$ for energies in a neighbourhood of $p = 0$, satisfying the following properties:
\be
&& \text{  (i) } \frac{\partial \tchi}{\partial p} = 0 \text{ for } p
\text{ near }0, \\
&& \text{ (ii) } \tchi \equiv 0 \text{ outside a neighbourhood of } \Sigma \text{ in the direction of } - \nu, \\
&& \text{(iii) } \tchi \equiv 1 \text{ in a neighbourhood of }
\Sigma , 
\ee
Choose also an energy cutoff $\psi \in \Ci_c ( \reals )$, $\psi(t) \equiv 1$
near $0$, $\psi \equiv 0$ away from $0$, and let $\chi = \psi (p) \tchi$.    
For a fixed constant $C>0$, let $\{ u_j \}$ be
the orthonormalized set of eigenfunctions of $P(h)$ satisfying
\be
P(h) u_j = E_j(h) u_j, \,\,\, E_j(h) = o(1),
\ee
and define the quantum flux product on $\{u_j\}$ by
\be
\lll u_j, u_k \rrr_{QF} = \lll \frac{i}{h} [P, \chi ]_+ u_j, u_k
\rrr_{L^2(X)},
\ee
where $[P, \chi]_+$ denotes the part of the commutator supported near
$\Sigma$ in the direction of $\nu$.  See Figure \ref{F:transversal}.
  \begin{figure}
\hfill
\centerline{%% Creator: Inkscape 1.3 (0e150ed6c4, 2023-07-21), www.inkscape.org
%% PDF/EPS/PS + LaTeX output extension by Johan Engelen, 2010
%% Accompanies image file 'Sigma-transversal-1.pdf' (pdf, eps, ps)
%%
%% To include the image in your LaTeX document, write
%%   \input{<filename>.pdf_tex}
%%  instead of
%%   \includegraphics{<filename>.pdf}
%% To scale the image, write
%%   \def\svgwidth{<desired width>}
%%   \input{<filename>.pdf_tex}
%%  instead of
%%   \includegraphics[width=<desired width>]{<filename>.pdf}
%%
%% Images with a different path to the parent latex file can
%% be accessed with the `import' package (which may need to be
%% installed) using
%%   \usepackage{import}
%% in the preamble, and then including the image with
%%   \import{<path to file>}{<filename>.pdf_tex}
%% Alternatively, one can specify
%%   \graphicspath{{<path to file>/}}
%% 
%% For more information, please see info/svg-inkscape on CTAN:
%%   http://tug.ctan.org/tex-archive/info/svg-inkscape
%%
\begingroup%
  \makeatletter%
  \providecommand\color[2][]{%
    \errmessage{(Inkscape) Color is used for the text in Inkscape, but the package 'color.sty' is not loaded}%
    \renewcommand\color[2][]{}%
  }%
  \providecommand\transparent[1]{%
    \errmessage{(Inkscape) Transparency is used (non-zero) for the text in Inkscape, but the package 'transparent.sty' is not loaded}%
    \renewcommand\transparent[1]{}%
  }%
  \providecommand\rotatebox[2]{#2}%
  \newcommand*\fsize{\dimexpr\f@size pt\relax}%
  \newcommand*\lineheight[1]{\fontsize{\fsize}{#1\fsize}\selectfont}%
  \ifx\svgwidth\undefined%
    \setlength{\unitlength}{194.21146774bp}%
    \ifx\svgscale\undefined%
      \relax%
    \else%
      \setlength{\unitlength}{\unitlength * \real{\svgscale}}%
    \fi%
  \else%
    \setlength{\unitlength}{\svgwidth}%
  \fi%
  \global\let\svgwidth\undefined%
  \global\let\svgscale\undefined%
  \makeatother%
  \begin{picture}(1,0.98230949)%
    \lineheight{1}%
    \setlength\tabcolsep{0pt}%
    \put(0,0){\includegraphics[width=\unitlength,page=1]{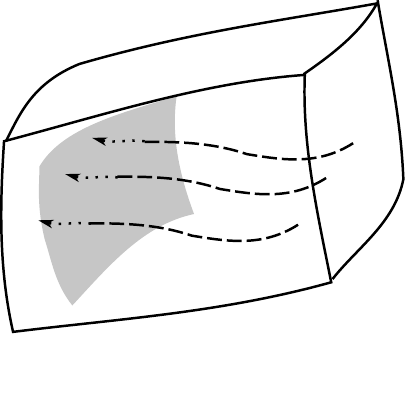}}%
    \put(0.6726567,0.17932606){\color[rgb]{0,0,0}\makebox(0,0)[lt]{\lineheight{1.25}\smash{\begin{tabular}[t]{l}$p^{-1}(0)$\end{tabular}}}}%
    \put(0.03708621,0.34032218){\color[rgb]{0,0,0}\makebox(0,0)[lt]{\lineheight{1.25}\smash{\begin{tabular}[t]{l}$\Sigma$\end{tabular}}}}%
    \put(0,0){\includegraphics[width=\unitlength,page=2]{Sigma-transversal-1.pdf}}%
    \put(0.56735141,0.05510502){\color[rgb]{0,0,0}\makebox(0,0)[lt]{\lineheight{1.25}\smash{\begin{tabular}[t]{l}$\chi$\end{tabular}}}}%
  \end{picture}%
\endgroup%
}
\caption{\label{F:transversal} A  sketch of the energy slab $\{ p = 0 \}$ with a hypersurface $\Sigma \subset \{ p = 0 \}$ which is transversal to the $H_p$ flow.  The part of the function $\chi$ 
near $\Sigma$ is sketched below.}
\hfill
\end{figure}

For concreteness, if $\Sigma$ is symplectic, in local symplectic coordinates, we can assume the symbol of $p$ is $p = \xi_1$ and, further introducing Fermi coordinates, our hypersurface $\Sigma \subseteq \{ x_1 = 0 \}$.  Then our function $\chi$ can be chosen of the form
\[
\chi = \tchi(x_1) \psi(\xi_1).
\]
The purpose of this is that $\psi$ localizes to energes $p$ near $0$.  The function $\tchi(x_1)$ has the property that near $\Sigma$, $[P, \tchi(x_1)] = P \tchi(x_1) = \tchi'(x_1)$ approximates a Dirac mass on $\Sigma$.

% $\clubsuit$ This Lemma needs to be fixed to only use $\lll u_j, u_j
% \rrr$, since I don't think it's correct with different $j$ and $k$,
% and all we use anyway is with $k = j$. $\clubsuit$
\begin{lemma}
\label{qf-lemma}
The quantum flux product $\lll \cdot, \cdot \rrr_{QF}$ induces a norm
which does not depend on
the choice of $\chi$, modulo $\O(h^\infty)$ as $h \to 0$.  Furthermore, if $\Sigma$ is
the boundary of an open submanifold of $\{p=0\}$, then $\lll u_j, u_j
\rrr_{QF} = \O(h^\infty)$, as $h \to 0$.
\end{lemma}
\begin{proof}
If $\chi_2$ is a cutoff also satisfying the conditions
above, then
\be
[P, \chi - \chi_2]_+ = [P, \chi-\chi_2] + \O(h^\infty),
\ee
since we may take $\chi$ and $\chi_2$ to be equal away from $\Sigma$.  See Figure \ref{F:lemma}.
 \begin{figure}
\hfill
\centerline{%% Creator: Inkscape 1.3 (0e150ed6c4, 2023-07-21), www.inkscape.org
%% PDF/EPS/PS + LaTeX output extension by Johan Engelen, 2010
%% Accompanies image file 'chi-chi2.pdf' (pdf, eps, ps)
%%
%% To include the image in your LaTeX document, write
%%   \input{<filename>.pdf_tex}
%%  instead of
%%   \includegraphics{<filename>.pdf}
%% To scale the image, write
%%   \def\svgwidth{<desired width>}
%%   \input{<filename>.pdf_tex}
%%  instead of
%%   \includegraphics[width=<desired width>]{<filename>.pdf}
%%
%% Images with a different path to the parent latex file can
%% be accessed with the `import' package (which may need to be
%% installed) using
%%   \usepackage{import}
%% in the preamble, and then including the image with
%%   \import{<path to file>}{<filename>.pdf_tex}
%% Alternatively, one can specify
%%   \graphicspath{{<path to file>/}}
%% 
%% For more information, please see info/svg-inkscape on CTAN:
%%   http://tug.ctan.org/tex-archive/info/svg-inkscape
%%
\begingroup%
  \makeatletter%
  \providecommand\color[2][]{%
    \errmessage{(Inkscape) Color is used for the text in Inkscape, but the package 'color.sty' is not loaded}%
    \renewcommand\color[2][]{}%
  }%
  \providecommand\transparent[1]{%
    \errmessage{(Inkscape) Transparency is used (non-zero) for the text in Inkscape, but the package 'transparent.sty' is not loaded}%
    \renewcommand\transparent[1]{}%
  }%
  \providecommand\rotatebox[2]{#2}%
  \newcommand*\fsize{\dimexpr\f@size pt\relax}%
  \newcommand*\lineheight[1]{\fontsize{\fsize}{#1\fsize}\selectfont}%
  \ifx\svgwidth\undefined%
    \setlength{\unitlength}{238.02420616bp}%
    \ifx\svgscale\undefined%
      \relax%
    \else%
      \setlength{\unitlength}{\unitlength * \real{\svgscale}}%
    \fi%
  \else%
    \setlength{\unitlength}{\svgwidth}%
  \fi%
  \global\let\svgwidth\undefined%
  \global\let\svgscale\undefined%
  \makeatother%
  \begin{picture}(1,0.28098257)%
    \lineheight{1}%
    \setlength\tabcolsep{0pt}%
    \put(0,0){\includegraphics[width=\unitlength,page=1]{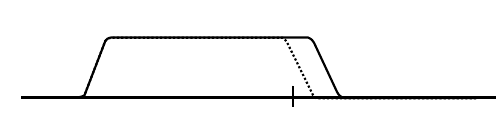}}%
    \put(-0.00112912,0.22631531){\color[rgb]{0,0,0}\makebox(0,0)[lt]{\lineheight{1.25}\smash{\begin{tabular}[t]{l}$\chi =\chi_2$\end{tabular}}}}%
    \put(0.52681514,0.01229877){\color[rgb]{0,0,0}\makebox(0,0)[lt]{\lineheight{1.25}\smash{\begin{tabular}[t]{l}$\Sigma$\end{tabular}}}}%
  \end{picture}%
\endgroup%
}
\caption{\label{F:lemma} A  sketch of two functions $\chi$ and $\chi_2$ satisfying the assumptions of the definition.  Since the definition of $[\cdot, \cdot]_+$ only uses the part of $\chi$ or $\chi_2$ near $\Sigma$, without loss in generality $\chi = \chi_2$ away from $\Sigma$.}
\hfill
\end{figure}
We then have, using $P^* = P$,
\be
\lefteqn{ \lll  \frac{i}{h} [P, \chi - \chi_2]_+ u_j, u_j
\rrr} \\
 & = & \lll \frac{i}{h} P (\chi - \chi_2) u_j, u_j
\rrr \\
&& \quad - \frac{i}{h} E_j(h) \lll   (\chi - \chi_2)
u_j, u_j \rrr + \O(h^\infty)\\
& = &  \frac{i}{h} (E_j(h) - E_j(h)) \lll (\chi - \chi_2)
u_j, u_j \rrr + \O(h^\infty) \\
& = & \O(h^\infty).
\ee
% If $(E_k-E_j) = o(h)$ the above expression is $o(1)$.  If $|E_k - E_j|
% \geq h/C$ for some constant $C>0$, then the frequency localization of
% eigenfunctions in phase space again implies the above expression is $o(1)$. 

If $\Sigma$ is the boundary of an open submanifold $S$, we can choose
$\chi \equiv 0$ on $S$, $\chi \equiv 1$ outside a neighbourhood of
$S$.  Then $ [P, \chi ]_+ = [P, \chi]$ and a similar calculation shows $\lll u_j, u_j
\rrr_{QF} = \O(h^\infty)$ for all $j$, as $h \to 0$.
\end{proof}

%\cs Don't we need $P^* = P$ in second last line of the above proof? \cs \\

In the main theorem, we are interested in symbols which are microlocally
restricted to $\Sigma$.  If the $H_p$ flow is transversal to $\Sigma$ in a 
neighbourhood of a point $m \in \Sigma$, this restriction is
independent of the values of a symbol $a$ in the direction of the
flow.  If the flow is tangential to $\Sigma$ at a point $m \in \Sigma$, then the limit measure
is zero at $m$ anyway.  Hence in this case we lose nothing by assuming $a$ is flow
invariant close to $\Sigma$.  This motivates the following definition:
\ben
\Psi_{h, \Sigma, \text{fi}}^{0,0}(X) & = & \{ A \in \Psi_{h}^{0,0}(X) :
\WF A \text{ is close to } \Sigma, \,\, \label{E:fi-def}\\
&& \quad \WF A
\cap \partial \Sigma = \emptyset, \text{ and } \WF [P,A] = \emptyset \text{ near } \Sigma \}.
\een
The point is that such operators have well-defined restrictions to
$\Sigma$ and commute with $P$ microlocally so that $\lll A u_j, u_j
\rrr_{QF}$ is still meaningful. Specifically, for $A \in \Psi_{h, \Sigma, \text{fi}}^{0,0}(X)$ we define

\begin{equation} \label{AQF}
\langle A u_j, u_j \rangle_{QF} := \lll \frac{i}{h} [P, \chi ]_+  A u_j, u_k
\rrr_{L^2(X)}, 
\end{equation}

%%%%%%%%%%%%%%%%%%%%%%%%%%%%%%%%%%%%%%%%%%%%%%%%%%%%%%%%%%%%%%%%%%%%%%%%%%%%%%%%%%%%%%%%%%%%%%%%%%%%%%%%%%%%%%%%%%%%%%%%%%%%%%%%%%%%%%%%%%%%%%%%%%%%%%%%%%%%%%%%%%%%%%%%%%%%%%%%%%%%%%%%%%%%%%%%%%%%%%%%%%%%%%%%%%%%%%%%%%%%%%%%%%%%%%%%%%%%%%%%%%%%%%%%%%%%%%%%%%%%%%%%%%%%%%%%%%%%%%%%%%%%%%%%%%%%%%%%%%%%%%%%%%%%%%%%%%%%%%%%%%%%%%%%%%%%%%%%%%%%%%%%%%%%%%%%%%%%%%%%%%%%%%%%%%%%

\section{Proof of Theorem \ref{theorem-1}}

In this section we prove Theorem \ref{theorem-1}.  Of course
all of the calculations are trivial if every quantity goes
to zero, so in
subsection \ref{example-section} we provide an example where the
quantum flux norm does not vanish.

\subsection{Proof of Theorem \ref{theorem-1}}

The proof of Theorem \ref{theorem-1} follows immediately from the
following proposition and a consideration of orientation.

\begin{proposition}
\label{prop-1}
Let $\{ u_j \}$ be an orthonormal sequence of quantum ergodic
eigenfunctions of $P(h)$,
$P(h) u_j = E_j(h) u_j$, $|E_j(h)| = o(1)$.
Assume $A \in \Psi_h^{0,0}(X)$, $\WF A$ is close to $\Sigma$, $\WF A
\cap \partial \Sigma = \emptyset$ if $\partial \Sigma \neq \emptyset$,
and $\WF [P,A]$ is away from $\Sigma$.  Then there exists a signed Radon measure
$\alpha$ on $\Sigma$ such that 
\be
\lll A u_j, u_j \rrr_{QF} \rightarrow \frac{1}{\mu ( \{p=0\})} \int_\Sigma
\sigma_h(A) d \alpha.
\ee
Here $\lll \cdot, \cdot \rrr_{QF} = \lll \cdot, \cdot \rrr_{QF( \Sigma)}$ denotes the quantum flux norm
associated to $\Sigma$ and $\mu$ is the Liouville measure on $\{p=0\}$.
\end{proposition}

Choose an orientation for $\Sigma$, and let $\nu$ be the normal field
to $\Sigma$ in $\{p=0\}$.  Let $\beta \in \Ci(T^*X)$ be a defining
function for $\Sigma$ which is increasing in the direction of $\nu$, normalized by $\| \nabla \beta|_\Sigma \|= 1$.
We define the measure $\alpha$ on $\Sigma$ by the following:
\be
d \alpha = \left\{ \begin{array}{l} \frac{1}{(n-1)!} (H_p \beta)|_\Sigma
    (\omega|_\Sigma)^{n-1}, \,\,\, H_p \beta \neq 0, \\
    0, \,\,\, H_p \beta = 0, \end{array} \right. 
\ee
where $\omega$ is the symplectic form on $T^*X$. % and $\sgn = \pm 1$ is
% the sign function.
The definition of 
$\alpha$ does not depend on any choice of coordinates, and is
independent of the choice of $\beta$ up to the choice of orientation
for $\Sigma$.

\begin{proof}[Proof of Proposition \ref{prop-1}]
The first reduction is to assume $\mu ( \{p=0\})=1$, which can be
brought about by multiplying $p$ by a constant.  The statement of the proof is invariant under coordinate changes.
Hence, with the aid of a partition of unity, it suffices to prove the
theorem in local coordinates.  Taking a refinement of the partition,
if necessary, we assume we are working in a coordinate patch in which
$H_p \beta$ does not change sign.  There are two cases to consider.

{\bf Case 1: $H_p \beta \neq 0$ on a coordinate patch.}  In this case, taking a further
refinement if necesssary, and recalling that the quantum flux norm is independent of the choice of $\chi$, we take $\chi = \tchi ( \beta / \epsilon)$, $\epsilon>0$ for $\tchi$ satisfying the assumptions of $\chi$.  We write locally near $\Sigma$,
\[
\sigma_h\left(\frac{i}{h}[P, \chi]_+\right) =  \frac{1}{\epsilon} \tchi'(\beta/\epsilon) H_p(\beta), 
\]
and
\be
 \int_{\{p=0\}}
  \sigma_h\left(\frac{i}{h}[P, \chi]_+ A\right) d \mu & = &
 \int_{-C\epsilon}^{C \epsilon} \int_\Sigma \frac{1}{\epsilon} \tchi'(\beta/\epsilon) H_p(\beta) a (x, \xi) d \mu + O(h)\\
& = & \frac{1}{(n-1)!} \int_\Sigma a (H_p \beta|_\Sigma)|(\omega|_\Sigma)^{n-1}| + \O( h
+ \epsilon ).
\ee
Taking the limit as $\epsilon \to 0$ proves the proposition in the case $H_p \beta \neq 0$.

{\bf Case 2: $H_p \beta = 0$ on a coordinate patch.}  If $H_p \beta = 0$ as in Figure \ref{F:tangential}, we take the same
coordinate system as in Case 1, but after a linear symplectic change
of variables, $\Sigma \subset \{ x_2 = \xi_1
=0 \}$, say.  Taking $\chi = \chi(x_2)$ in a neighbourhood of $\WF A$
yields
\be
\sigma_h\left(\frac{i}{h}[P, \chi]_+\right) & = & \partial_{x_1}
\chi(x_2) \\
& = & 0.
\ee
  \begin{figure}
\hfill
\centerline{%% Creator: Inkscape 1.3 (0e150ed6c4, 2023-07-21), www.inkscape.org
%% PDF/EPS/PS + LaTeX output extension by Johan Engelen, 2010
%% Accompanies image file 'Sigma-tangential-1.pdf' (pdf, eps, ps)
%%
%% To include the image in your LaTeX document, write
%%   \input{<filename>.pdf_tex}
%%  instead of
%%   \includegraphics{<filename>.pdf}
%% To scale the image, write
%%   \def\svgwidth{<desired width>}
%%   \input{<filename>.pdf_tex}
%%  instead of
%%   \includegraphics[width=<desired width>]{<filename>.pdf}
%%
%% Images with a different path to the parent latex file can
%% be accessed with the `import' package (which may need to be
%% installed) using
%%   \usepackage{import}
%% in the preamble, and then including the image with
%%   \import{<path to file>}{<filename>.pdf_tex}
%% Alternatively, one can specify
%%   \graphicspath{{<path to file>/}}
%% 
%% For more information, please see info/svg-inkscape on CTAN:
%%   http://tug.ctan.org/tex-archive/info/svg-inkscape
%%
\begingroup%
  \makeatletter%
  \providecommand\color[2][]{%
    \errmessage{(Inkscape) Color is used for the text in Inkscape, but the package 'color.sty' is not loaded}%
    \renewcommand\color[2][]{}%
  }%
  \providecommand\transparent[1]{%
    \errmessage{(Inkscape) Transparency is used (non-zero) for the text in Inkscape, but the package 'transparent.sty' is not loaded}%
    \renewcommand\transparent[1]{}%
  }%
  \providecommand\rotatebox[2]{#2}%
  \newcommand*\fsize{\dimexpr\f@size pt\relax}%
  \newcommand*\lineheight[1]{\fontsize{\fsize}{#1\fsize}\selectfont}%
  \ifx\svgwidth\undefined%
    \setlength{\unitlength}{194.21146774bp}%
    \ifx\svgscale\undefined%
      \relax%
    \else%
      \setlength{\unitlength}{\unitlength * \real{\svgscale}}%
    \fi%
  \else%
    \setlength{\unitlength}{\svgwidth}%
  \fi%
  \global\let\svgwidth\undefined%
  \global\let\svgscale\undefined%
  \makeatother%
  \begin{picture}(1,0.98230949)%
    \lineheight{1}%
    \setlength\tabcolsep{0pt}%
    \put(0,0){\includegraphics[width=\unitlength,page=1]{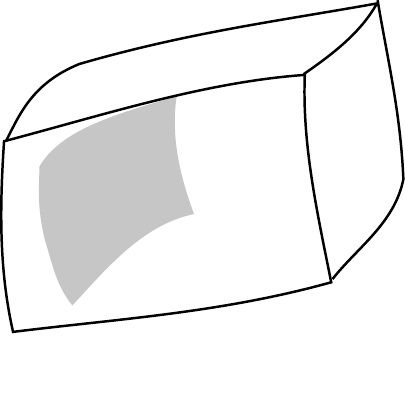}}%
    \put(0.6726567,0.17932606){\color[rgb]{0,0,0}\makebox(0,0)[lt]{\lineheight{1.25}\smash{\begin{tabular}[t]{l}$p^{-1}(0)$\end{tabular}}}}%
    \put(0.07460454,0.2802929){\color[rgb]{0,0,0}\makebox(0,0)[lt]{\lineheight{1.25}\smash{\begin{tabular}[t]{l}$\Sigma$\end{tabular}}}}%
    \put(0,0){\includegraphics[width=\unitlength,page=2]{Sigma-tangential-1.pdf}}%
    \put(0.56735141,0.05510502){\color[rgb]{0,0,0}\makebox(0,0)[lt]{\lineheight{1.25}\smash{\begin{tabular}[t]{l}$\chi$\end{tabular}}}}%
    \put(0,0){\includegraphics[width=\unitlength,page=3]{Sigma-tangential-1.pdf}}%
    \put(0.37049324,0.33980839){\color[rgb]{0,0,0}\makebox(0,0)[lt]{\lineheight{1.25}\smash{\begin{tabular}[t]{l}$H_p \chi \sim 0$\end{tabular}}}}%
  \end{picture}%
\endgroup%
}
\caption{\label{F:tangential} A  sketch of the energy slab $\{ p = 0 \}$ with a hypersurface $\Sigma \subset \{ p = 0 \}$ which is tangential to the $H_p$ flow.  The part of the function $\chi$ 
near $\Sigma$ is sketched below.  In this case $H_p \chi = 0$ so the quantum flux is trivial.}
\hfill
\end{figure}

The proof of the proposition is concluded by applying the conclusion
of the interior quantum ergodicity theorem  with $\frac{i}{h}[P, \chi]_+
  A$ in place of $A$.
\end{proof}

\begin{remark}
To get the signed measure $\alpha$ on all of $\Sigma$, a combination of the two cases above is used.  Using a globally defined $\beta$ fixes the sign of $H_p \beta$ on each connected component of $\{ H_p \beta \neq 0 \}$.

\end{remark}

%%%%%%%%%%%%%%%%%%%%%%%%%%%%%%%%%%%%%%%%%%%%%%%%%%%%%%%%%%%%%%%%%%%%%%%%%%%%%%%%%%%%%%%%%%%%%%%%%%%%%%%%%%%%%%%%%%%%%%%%%%%%%%%%%%%%%%%%%%%%%%%%%%%%%%%%%%%%%%%%%%%%%%%%%%%%%%%%%%%%%%%%%%%%%%%%%%%%%%%%%%%%%%%%%%%%%%%%%%%%%%%%%%%%%%%%%%%%%%%%%%%%%%%%%%%%%%%%%%%%%%%%%%%%%%%%%%%%%%%%%%%%%%%%%%%%%%%%%%%%%%%%%%%%%%%%%%%%%%%%%%%%%%%%%%%%%%%%%%%%%%%%%%%%%%%%%%%%%%%%%%%%%%%%%%%%

\section{Some Examples}
\label{example-section}
In this section, we first provide an example of a symbol $p$ satisfying the
assumptions of Theorem \ref{theorem-1} whose energy
surface $\{p=0\}$ admits a cross section $\Sigma\subset \{p=0\}$ with
nontrival quantum flux product.  % We then provide an example of a
% hypersurface Poincar\'e map.  
We then revisit the definition of the quantum flux norm in the case of
Laplace eigenfunctions in a flat Euclidean domain and demonstrate that
our main result reproves \cite[Theorem 1]{CTZ}.

\subsection{Example of non-trivial quantum flux}
Let us first provide an example for Theorem \ref{theorem-1}.  We know from Lemma \ref{qf-lemma}
that $\Sigma$ must be chosen differently from the boundary of an open
submanifold.  

Let $X = \TT^2$ which for simplicity, we give the coordinates
$(x_1,x_2) \in [-1,1] \times [-1,1] = \reals^2 / 2 \ZZ^2$.  Let $V \in
\Ci_c(X)$, $0 \leq V \leq 1$, be a radial bump potential with support
contained in $B(0,1/4)$.  Then for $E \sim 1/2$, $p = \xi_1^2 +
\xi_2^2 + V(x) - E$ satisfies the assumptions of Theorem
\ref{theorem-1}.  Let 
\be
\Sigma = \left\{ x_1 = 1/2,\, -1 \leq x_2 \leq 1,\, \xi_1 = \sqrt{ E -
  \xi_2^2}, \,0 \leq \xi_2 \leq E^{1/2} \right\},
\ee
so that
\be
\omega|_\Sigma & = & (d\xi_1 \wedge dx_1 + d \xi_2 \wedge dx_2
)|_\Sigma \\
& = & d \xi_2 \wedge dx_2.
\ee
Observe that in a neighbourhood of $\Sigma$, 
\be
H_p = 2 \xi_1 \partial_{x_1} + 2 \xi_2 \partial_{x_2},
\ee
and a choice of $\beta$ is $\beta = x_1 - 1/2$,
for simplicity we take, for example, 
\be
a = \sigma_h(A) = g(\xi_2).
\ee
To ensure $\WF A$ does not meet $\partial \Sigma$, we take $g$
compactly supported in the interval, say, $E^{1/2}/4 \leq \xi_2 \leq 3
E^{1/2}/4$.  If $g \geq 0$, $g>0$ on some interval, we calculate
\be
\int_\Sigma a (H_p \beta) \omega|_{\Sigma} & = & 2 \int_{x_2 = -1}^{x_2 = 1}
\int_{\xi_2 = 0}^{\xi_2 = E^{1/2}} g( \xi_2) d \xi_2 dx_2 \\
& >& 0.
\ee

%%%%%%%%%%%%%%%%%%%%%%%%%%%%%%%%%%%%%%%%%%%%%%%%%%%%%%%%%%%%%%%%%%%%%%%%%%%%%%%%%%%%%%%%%%%%%%%%%%%%%%%%%%%%%%%%%%%%%%%%%%%%%%%%%%%%%%%%%%%%%%%%%%%%%%%%%%%%%%%%%%%%%%%%%%%%%%%%%%%%%%%%%%%%%%%%%%%%%%%%%%%%%%%%%%%%%%%%%%%%%%%%%%%%%%%%%%%%%%%%%%%%%%%%%%%%%%%%%%%%%%%%%%%%%%%%%%%%%%%%%%%%%%%%%%%%%%%%%%%%%%%%%%%%%%%%%%%%%%%%%%%%%%%%%%%%%%%%%%%%%%%%%%%%%%%%%%%%%%%%%%%%%%%%%%%%

\subsection{Laplace Eigenfunctions and the Proof of Corollary \ref{C:cor}}
\label{SS:cor}

%\subsection{Quantum Flux Norm Redux}
In this subsection, we re-interpret \cite[Theorem 1]{CTZ} in terms of
the quantum flux norm.  Let $X$ be a compact Riemannian manifold with
or without boundary.
We want to define a reasonable quantum flux norm which is
related to the computation in \cite[Theorem 1]{CTZ}.  In this case, we let $P =
-h^2 \Delta -1$, so that $\{ p = 0 \} = S^* X$.  We assume
throughout this subsection that $u_j$ is a quantum ergodic sequence
of orthonormalized eigenfunctions $P u_j = 0$.

If $H \subset
X$ is the boundary of an open submanifold $X_H \subset
X$, we cannot just blindly take $\Sigma = S^*_H X \subset
\{p = 0 \}$ as our codimension $1$ submanifold of $\{p = 0 \}$, since
it is the boundary of $S^* X_H$ in $ \{ p = 0 \}$ and by Lemma
\ref{qf-lemma} the quantum flux norm must then vanish.  Instead, we
define the incoming/outgoing quantum flux norms in the following
fashion.  We will work locally in normal coordinates to $H$ as usual.
Then $H = \{ x_n = 0 \}$, and $-h^2 \Delta = (hD_n)^2 + R(x', x_n, hD')$.  We
define the outgoing Quantum Flux norm, which is associated to
$S^*_{H+} X = \{(x, \xi) \in S^*_H X  : \xi_n >0 \}$.
The
construction of the quantum flux norm above requires a cutoff
$\chi(x_n/\epsilon)$ which is $0$ for $x_n< - \epsilon $ and $1$ for $x_n \geq 0$, as well
as a cutoff $\psi = \psi(p)$ to a neighbourhood of the set $\{p = 0 \}$.  Then
if $(i/h) [P, \chi \psi]_+\mathbbm{1}_{\xi_n >0 }$ is the part of the
commutator supported in the
direction of increasing $x_n$, and associated to {\it positive} normal
directions ($\xi_n >0$), $(i/h)[P,\chi \psi ]_+\mathbbm{1}_{\xi_n >0 } =
(i/h)[P, \chi \psi ]  \mathbbm{1}_{\xi_n>0} \tchi( x_n/\delta)$, where
$\delta> \epsilon >0$ is a small parameter and $\tchi = 1$ in a
neighbourhood of $0$.  We define the outgoing/incoming quantum flux norms
associated to $H$ by
\begin{equation} \label{QEQF1}
\lll u_j, u_j \rrr_{\text{QF}(H\pm)} := \lim_{\epsilon \to 0} \lll \frac{i}{h}
[P, \chi(x_n/\epsilon)]  \mathbbm{1}_{\pm \xi_n>0} \tchi( x_n/\delta) u_j, u_j \rrr.
\end{equation}
% Similarly we define the incoming quantum flux norm:
% \[
% \lll \phi, \phi \rrr_{\text{QF}^-} = \lll \frac{i}{h}
% \mathbbm{1}_{hD_n<0}[P, \chi] hD_n \phi, \phi \rrr.
% \]

Our first observation is
that in the very special case that the sequence of eigenfunctions is
quantum ergodic, it is no longer necessary to assume the
pseudodifferential weight $A$ is flow-independent, at the expense that
the quantum flux norm is no longer independent of the choice of
$\chi$.  However, this is not a problem if we interpret the quantum
flux norm with a limit as $\epsilon \to 0$ as used in the proofs of
Theorems \ref{theorem-1} in the present article and \cite[Theorem 1]{CTZ}.  That is, with the
outgoing/incoming quantum flux norms defined above, if
$\chi_2(x_n/\epsilon)$ satisfies the same assumptions as
$\chi(x_n/\epsilon)$, then the commutator
\[
(i/h) [ P, \chi(x_n/\epsilon) - \chi_2(x_n/\epsilon)]_+ 
\]
is an honest commutator, which is supported in a region of size $ \epsilon$ in the $x$-space,
and by integration by parts, and interior quantum ergodicity we have
\[
  \lim_{\epsilon \to 0} \lll (i/h) [ P, \chi(x_n/\epsilon) - \chi_2(x_n/\epsilon)]_+ A u_j,
u_j \rrr = o(1)
\]
as $h \to 0$ independent of the choice of $A \in \Psi^0(H)$.

Written out in a coordinate independent fashion, we write
\[
\lll u_j, u_j \rrr_{\text{QF}(H\pm)} = \lll \frac{i}{h} [P,
\chi(\beta/\epsilon) \psi(P)]_+ \mathbbm{1}_{\pm h D_\nu} u_j, u_j \rrr,
\]
where $\beta$ is a normalized defining function for $H$ in the
$x$-space, and $\nu = \nabla\beta$ is the outward pointing unit normal
field in the $x$-space.

We observe that $\pm H_p \beta \mathbbm{1}_{\pm \xi_\nu >0} >0$, so
that we get
\[
\lll u_j, u_j \rrr_{\text{QF}(H+)} + \lll u_j,u_j \rrr_{\text{QF}(H-)} = \O(h^\infty),
\]
in accordance with Lemma \ref{qf-lemma}, since the union of the incoming and
outgoing regions associated to $H$ are the boundary of a submanifold
of $\{p = 0 \}$.  On the other hand, since $hD_\nu \mathbbm{1}_{\pm h
  D_\nu >0}$ are distinct operators, and if $a(x', hD')$ is a
tangential operator to $H$,  
\begin{align} \label{key point}
& \lll hD_\nu a^w u_j, u_j \rrr_{\text{QF}(H+)} + \lll h D_\nu a^w u_j,u_j
\rrr_{\text{QF}(H-)} \\ \nonumber
& = \lll \mathbbm{1}_{ h
  D_\nu >0} hD_\nu a^w u_j, u_j \rrr_{\text{QF}(H)} + \lll \mathbbm{1}_{h
  D_\nu <0}h D_\nu a^w u_j,u_j
\rrr_{\text{QF}(H)} \\ \nonumber
& = \int_{X_H} \frac{i}{h}[P, \chi(\beta/\epsilon)\psi(P)] h
D_\nu a^w u_j \overline{u_j} dx + \O( \epsilon + \delta),
\end{align}
which is precisely the quantity computed in the proof of \cite[Theorem
1]{CTZ}.

Hence, Theorem \ref{theorem-1} applies in this case, and indeed, one
readily computes that $(H_p \beta)^w = \pm \sqrt{1 + h^2 \Delta_H}$
depending on the outgoing/incoming direction.  But then $\xi_\nu H_p
\beta = 1 + h^2 \Delta_H$ as in the statement of \cite[Theorem 1]{CTZ}.
 In other words, the explanation for the extra factor of
$2$ in the measure computed in \cite[Theorem 1]{CTZ} is because the
incoming and outgoing hemispheres in $S_H^* X$ restricted to $B^*
H$ give {\it two} copies of the same measure.

% The intuition behind the $\pm$ signs is that in the original quantum
% flux norm, the cutoff is supposed to be increasing in the direction of
% the normal vector.  Here $\chi \xi_n$ is increasing for the outgoing
% quantum flux norm, but decreasing for the incoming quantum flux norm.
% At first it seems simple-minded to just take a partition of $\pm
% \xi_n>0$, but the observation we make is that of course
% \[
% \lll \phi, \phi  \rrr_{\text{QF}^+} + \lll \phi, \phi
% \rrr_{\text{QF}^-} = 0,
% \]
% while
% \[
% \lll \phi, \phi  \rrr_{\text{QF}^+} - \lll \phi, \phi
% \rrr_{\text{QF}^-} = \frac{i}{h} \lll [P, \chi_1(x_n/\epsilon) hD_n] \phi, \phi \rrr +
% \O( \epsilon),
% \]
% as $\epsilon \to 0$ provided the $\phi$ are quantum ergodic.  Here we
% have used $\chi_1$ with the properties in the previous section.  

%%%%%%%%%%%%%%%%%%%%%%%%%%%%%%%%%%%%%%%%%%%%%%%%%%%%%%%

{\section{2 microlocal quantum flux: Proof of Theorem \ref{2microlocalQER} }
\label{s:FBIDescription}

\subsubsection{ Grauert tubes and analytic $h$-pseudodifferential calculus}  Let $X$ be a compact, closed, real-analytic manifold of dimension $n$  and $X^{\C}$  denote a Grauert tube complex thickening of $X$ which is a  totally real submanifold. By Bruhat-Whitney,  there exists  a {\em maximal} Grauert tube radius $\tau_{\max} >0$ \cite{Z} such that for any $ \tau \leq \tau_{\max}$, the complex manifold $X^{\C}$ can be identified with $B^*_{\tau} := \{ (x, \xi) \in T^*X; \sqrt{\rho}(x,\xi) \leq \tau \}$ where $\sqrt{2\rho} = |\xi|_g$ is the exhaustion function using the complex geodesic exponential map $ \kappa : B^*_{\tau} \rightarrow \tilde{X}$ with $\kappa(x,\xi) = \exp_{x}( i \xi).$  From now on, we fix $\tau \in (0, \tau_{\max}).$
Under this indentification, we will let $z$ denote local complex coordinates in $B^*_{\tau}$ and we recall that $B_{\tau}^*$ is also naturally a Kahler manifold with potential function $\rho$ with associated symplectic form
$ \partial \overline{\partial} \rho = \omega.$  The complex Kahler, symplectic and Riemannian structures are all linked via the isomorphism $\kappa: B_{\tau}^*X \to X^{\mathbb C}.$ Denoting the almost complex structure by $J: T^{\mathbb C} X^{\mathbb{C}} \to T^{\mathbb C} X^{\mathbb{C}},$
\begin{eqnarray} \label{grauertformula}
\omega = \partial \overline{\partial} \rho, \quad \omega = d \alpha, \,\, \alpha = \Im \overline{\partial} \rho,
\end{eqnarray}
  where the strictly plurisubharmonic function $\rho$ solves the homogeneous Monge-Ampere equation \cite{GSt} 
  
  $$ \big(  \overline{\partial} \partial \sqrt{\rho} \big)^n (z) = 0, \quad z \in X^{\mathbb C}.$$
  
  The $\kappa$-corresponding objects  on $B_{\tau}^*X$ are given by
  \begin{eqnarray} \label{grauertformula2}
  \kappa^*\omega  = \sum_{j=1}^n dx_j \wedge d\xi_j = \frac{1}{i} \kappa^* \partial \overline{\partial} \rho, \,\,\, \, \kappa^* \omega = d \alpha, \,\, \alpha = \sum_i \xi_i dx_i,\nonumber \\
 \kappa^*\rho(x,\xi) = \frac{1}{2} |\xi|_x^2 = \frac{1}{2} g^{ij}(x) \xi_i \xi_j. \hspace{1in}
 \end{eqnarray}
 
 In the following, we will freely identify $B_{\tau}^*X$ and $X^{\mathbb C}$ and drop reference to the isomorphism  $\kappa$ when the context is clear.
 
  The Riemannian Kahler metric $\tilde{g}$ on $B_\tau^* X$ associated with $\omega$ is given by
  $$ \tilde{g}(u,v) = \omega (u, Jv)$$
  where $J$ is the almost complex structure on $B_\tau^*X$ induced by $\kappa.$

 Fix $p_0 \in X$ and let $x: U \to {\mathbb R}^n$ be geodesic normal coordinates centered at $p_0$ with $x(p_0) = 0.$ Then since $J_{(p_0,\xi)} (\partial_{x_j}) = \partial_{\xi_j}$ and $J_{(p_0,\xi)}(\partial_{\xi_j}) = - \partial_{x_j}$ and the base metric $g^{ij}(x) = \delta_j^i + O(|x|^2)$ (in particular, $\partial_{x_j} g^{kl}(0) = 0$),  it follows that 
 $ \tilde{g}_{(p_0,\xi)} = |dx|_{(p_0,\xi)}^2 + |d\xi|_{(p_0,\xi)}^2.$   As a result,
 $$ \nabla_{\tilde{g}} \rho (p_0,\xi) =   \sum_{j=1}^{n} \partial_{x_j}  \big( \frac{1}{2} g^{kl}(x) \xi_k \xi_l \big) |_{x=0}  \partial_{x_j} + \partial_{\xi_j} \big( \frac{1}{2} g^{kl}(x) \xi_k \xi_l \big)|_{x=0} \partial_{\xi_j}   = \sum_{j} \xi_j \partial_{\xi_j}.$$
 
 Given that  $\omega$ in (\ref{grauertformula}) is non-degenerate with $\omega = d\alpha$ there is a unique invariant  vector field $X$ solving  $\iota_X \omega = \alpha.$ Moreover (see \cite{GSt} section 5), $X$ satisfies 
 $X \rho = 2 \rho$ and $\kappa^* X = \xi \cdot \partial_{\xi}$. Since $\xi \cdot \partial_{\xi}$ and $\nabla_{\tilde{g}} \kappa^*\rho$ are consequently both invariant vector fields on $B_\tau^*X$ which agree at $(p_0,\xi)$ in geodesic normal coordinates, they must agree in all local coordinates $x$ near $p_0$. Since $p_0 \in X$ is arbitrary, by making the usual identification of $B_\tau^*X$ with $X^{\C}$, it follows that
 
    \begin{equation} \label{grad}
\nabla_{\tilde{g}} \rho (x,\xi)  = \sum_{j=1}^n \xi_j \partial_{\xi_j}, \quad (x,\xi) \in B_{\tau}^*X.
\end{equation}

From (\ref{grad}) and the argument above, it also follows that
 \begin{equation} \label{gradnorm}
\| \nabla_{\tilde{g}} \rho \|_{\tilde{g}}^2 = 2 \rho.
\end{equation}

The associated Kahler Laplacian is 
$$ \Delta_{\overline{\partial}} = \overline{\partial}^* \overline{\partial} = 2 \Delta_{\tilde{g}}$$
where the latter denotes the Riemannian Laplacian with respect to $\tilde{g}$ on $B_\tau^*X.$
In the following, to simplify notation, we will write $\nabla := \nabla_{\tilde{g}}$ and $\Delta := \Delta_{\overline{\partial}}.$

Let $ -h^2 \Delta_{\overline{\partial} }: C^{\infty}_0(B^*_{\tau}) \to C^{\infty}_0(B^*_{\tau})$ denote the semiclassical Kahler Laplacian with $-h^2 \Delta_{\overline{\partial}} = - 2 h^2 \Delta_{\tilde{g}}.$  By possibly rescaling the semiclassical parameter $h$ we assume without loss of generality that the characteristic manifold
$ p^{-1}(0) \subset B^*_{\tau}.$

\subsubsection{Fermi coordinates near a hypersurface $\Sigma \subset B_{\tau}^*X.$} \label{FERMI}

Given a smooth oriented hypersurface $\Sigma \subset B_{\tau}^*X$ we let $(\beta',\beta): U_{\Sigma} \to {\mathbb R}^{2n}$ be
normalized Fermi coordinates in a tubular neighbourhood $U_{\Sigma}$ of $\Sigma$ with $\Sigma = \{ \beta = 0 \}$ and
$\partial_\beta$ the unit exterior normal to $\Sigma.$ In terms of these coordinates the conjugated Laplacian $ |\tilde{g}|^{1/4} \Delta_{\tilde{g}} |\tilde{g}|^{-1/4}$ can be written in the form
\begin{equation} \label{fermi expansion}
|\tilde{g}|^{1/4} (-h^2 \Delta_{\tilde{g}} ) \, |\tilde{g}|^{-1/4}= (h D_{\beta})^2 + R(\beta,\beta'; hD_{\beta'}),
\end{equation}

where $R(\beta,\beta',hD_{\beta'})$ is a second-order $h$-differential operator in the tangential $\beta'$-variables and 
$R(0,\beta', hD_{\beta'}) =  -h^2 \Delta_{\Sigma}$ where $\Delta_{\Sigma}$ is the Riemannian Laplacian on the hypersuface $\Sigma$  induced by the metric $\tilde{g}.$ In the following, we abuse notation and denote the conjugated Laplacian simply by $-h^2 \Delta_{\tilde{g}}$ and $| \tilde{g}|^{1/4} u_h$ by $u_h.$

\subsubsection{FBI transform} \label{SS:FBI}
 Let $U\subset T^*X$ be open. Following \cite{Sjo}, we say that $a \in S^{m,k}_{cla}(U)$  provided $a \sim h^{-m} (a_0 + h a_1 + \dots)$ in the sense that
\begin{equation*}
%\label{scsymbol} 
%\begin{gathered} 
\partial_{x}^k \partial_{\xi}^l \overline{\partial}_{(x,\xi)} a = O_{k,l}(1) e^{- \langle \xi \rangle/Ch}, \quad (x,\xi)\in U, 
\end{equation*}
and for $(x , \xi) \in U$,
\begin{equation*}
  \Big| a - h^{-m} \sum_{0 \leq j \leq \langle \xi \rangle/C_0 h} h^{j} a_j \Big| = O(1) e^{- \langle \xi \rangle/C_1 h},\quad
 |a_j| \leq C_0 C^{j} \, j ! \, \langle \xi \rangle^{k-j}.
% \end{gathered}
 \end{equation*}
We sometimes write $S^{m,k}_{cla}=S^{m,k}_{cla}(T^*X)$. The symbol $a \in S^{m,k}_{cla}$ is $h$-elliptic provided
$ |a(x,\xi)| \geq C h^{-m} \langle \langle \xi \rangle^k$ for all $(x,\xi) \in T^*X.$

 As in \cite{Sjo},  given an $h$-elliptic, semiclassical analytic symbol $a \in S^{3n/4,n/4}_{cla}(X \times (0,h_0]),$  we consider an intrinsic FBI transform $T(h):C^{\infty}(X) \to C^{\infty}(T^*X)$ of the form
\begin{equation} \label{FBI}
T u(x,\xi;h) = \int_{X} e^{i\phi(x,\xi,y)/h}  a(x,\xi,y,h) \tilde{\chi}( x, y) u(y) \, dy \end{equation}
In (\ref{FBI}), the cutoff $\tilde{\chi} \in C^{\infty}_{0}(X \times X)$ is supported in a  small fixed neighbourhood of $\text{diag}(X) = \{ (x,x,) \in X \times X \}.$ The phase function is required to satisfy $\phi(x,\xi,x) = 0, \, \partial_y \phi(x,\xi,x) = - \xi$ and
$$ \Im (\partial_y^2 \phi)(x,\xi, x) \sim | \langle \xi \rangle | \, Id.$$

In particular, it follows that the phase $\phi$ satisfies
\begin{eqnarray} \label{phase}
\Re \phi(x,\xi,y) = \langle x-y, \xi \rangle + O(|x-y|^2 \langle \xi \rangle), \nonumber \\
\Im \phi(x,\xi,y)  = \frac{1}{2} |x-y|^2 \big( 1 + O(|x-y|) \big)  \langle \xi \rangle.
\end{eqnarray}

% Fixing $T$ as in (\ref{FBI}), for any $b \in S^{0}_{cla},$ when convenient, we use the notation
%$T_b := b T.$

Given $T(h) :C^{\infty}(X) \to C^{\infty}(T^*X)$ it follows by an analytic stationary phase argument \cite{Sjo} that one can construct an operator $S(h): C^{\infty}(T^*X) \to C^{\infty}(X)$ of the form
\begin{equation} \label{left}
 S v(x;h) = \int_{T^*X} e^{-i  \, \overline{\phi(x,y,\xi)}  /h} d(x,y,\xi,h) v(y,\xi) \, dy d\xi \end{equation}
with $d \in S^{3n/4,n/4}_{cla}$ such that $S(h)$  is a left-parametrix for $T(h)$ in the sense that

\begin{equation} \label{leftparametrix}
S(h) T(h) = Id + R(h),\qquad\partial_{x}^{\alpha} \partial_{y}^{\beta} R(x,y,h) = O_{\alpha, \beta}(e^{-C/h}).
\end{equation}

We also note that with the normalizations in (\ref{FBI}) an application of analytic stationary phase as in (\ref{left}) shows that there exists $e_0 \in S^{0}_{cla}$  h-elliptic such that 
$$ T^*(h) e_0 T(h) = S(h)  T(h) +O(h)_{L^2 \to L^2},$$
and consequently, it follows that
$$ \| T u_h \|_{L^2}  \approx \| u_h \|_{L^2} \approx 1.$$
The standard left-quantized analytic $h$-pseudodifferential operator (h-psdo) $Op_h(a)$ with kernel 
$K_a(x,y) = (2\pi h)^{-n} \int e^{i \langle x-y,\xi \rangle/h} a(x,\xi,h) d\xi$ satisfies
\begin{equation} \label{antiwick}
Op_h(a) = S(h) a T(h) + O(h)_{L^2 \to L^2} = T^*(h) e_0\, a T(h)  + O(h)_{L^2 \to L^2} \end{equation}

For future reference, we recall here that in view of (\ref{phase}), by diferentiation under the integral in (\ref{FBI}) it follows that
\begin{eqnarray} \label{intertwine1}
hD_{x_j} T u(x,\xi) = \int_{X} e^{i\phi(x,\xi,y)/h} \big( \xi_j + O(|x-y|) \big) a(x,\xi,y) \tilde{\chi}(x,y) u(y) dy \nonumber \\
  = \xi_j \,Tu(x,\xi) + O_{L^2}(\sqrt{h})
\end{eqnarray}
since $|x-y| e^{i\phi(x,\xi,y)/h} = O(|x-y| e^{-\frac{C}{h} |x-y|^2 }) = O(\sqrt{h}).$
Similarily,
\begin{eqnarray} \label{intertwine2}
h D_{\xi_j} Tu(x,\xi) = \int_{X} e^{i\phi(x,\xi,y)/h} \big(x_j-y_j + O(|x-y|^2)\big) a(x,\xi,y) \tilde{\chi}(x,y) u(y) dy \nonumber \\
= O_{L^2}(\sqrt{h}).\end{eqnarray}

A similar argument  (see \cite{Sjo} and the appendix) shows that

\begin{eqnarray} \label{intertwine3}
\langle hD_{x_j} Tu_h, Tu_h \rangle_{L^2}  = \langle \xi_j Tu_h, Tu_h \rangle_{L^2} + O(h),\nonumber \\ \nonumber \\
\langle hD_{\xi_j} Tu_h, Tu_h \rangle_{L^2} = O(h).
\end{eqnarray}

It is convenient to choose here a particular FBI transform, $T_{hol}(h): C^{\infty}(X) \to C^{\infty}(B_{\tau}^*X)$ that is compatible with the complex structure in the Grauert tube $B_{\tau}^*X.$ This transform is readily described  in terms of the holomorphic continuation of the heat operator $e^{t \Delta_g}$ at time $t = h/2.$

We briefly recall here some background on the operator $T_{hol}(h): C^{\infty}(X) \to C^{\infty}(X_{\tau}^{\C})$ and refer the reader to \cite{GLS} and \cite{GT} for further details.

\subsubsection{Complexified heat operator on closed, compact manifolds} \label{heat}
Consider the heat operator of $(X,g)$ defined at time $h/2$ by $$E_{h}=e^{\frac{h}{2}\Delta_g}:C^{\infty}(X) \to C^{\infty}(X).$$

 By a result of Zelditch \cite[Section 11.1]{Z}, the maximal geometric tube radius $\tau_{\max}$ agrees with the maximal analytic tube radius in the sense that for all $ 0<\tau < \tau_{\max}$, all the eigenfunctions $\varphi_j$  extend holomorphically to $X_\tau^\C$ (see also \cite[Prop. 2.1]{GLS}). In particular, the kernel $E(\cdot,\cdot;h)$ admits a holomorphic extension to $B^*_{\tau} X \times B^*_{\tau}X$
for all $0<\tau < \tau_{\max}$  and $h \in (0,1)$, \cite[Prop. 2.4]{GLS}. We denote the complexification by $E_h^\C( \cdot, \cdot)$.
To recall asymptotics for $E^{\C}_h$ we note that
  the squared geodesic distance on $X$
$$r^2(\cdot, \cdot): X \times X \to {\mathbb R}$$
holomorphically continues in both variables to $X_{\tau}^{\C} \times X_{\tau}^{\C}$ in a straightforward fashion. 
More precisely, $0<\tau<\tau_{\max}$, there exists a connected open neighbourhood $\tilde \Delta \subset X_\tau^\C \times M_\tau^\C$ of the diagonal $\Delta \subset X\times X$ to which $r^2(\cdot, \cdot)$ can be holomorphically extended \cite[Corollary 1.24]{GLS}. We denote the extension
by $r_\C^2(\cdot,\cdot) \in \mathcal O(\tilde \Delta).$ Moreover, one can easily recover the exhaustion function $\sqrt{\rho_g}(z)$ from $r_{\C}$; indeed, 
$\rho_g(z)=-r^2_\C(z, \bar{z})$
for all $z \in B^*_{\tau}X$.

The basic asymptotic behaviour of $E_h^{\C}(z,y)$ with $(z,y) \in B^*_{\tau}X \times X$ is studied in \cite{GLS}. In particular,

\begin{equation}\label{hol heat}
E_h^\C(z,y)=e^{-\frac{r^2_\C(z,y)}{2h}} b^\C(z,y; h) + O (e^{-\beta/h}), \quad (z,y) \in B_{\tau}^*X \times X.
\end{equation}
Here, $ \beta>0$ is a constant depending on $(X,g,\tau)$ and
\begin{equation} \label{heat symbol}
b^\C \sim \sum_{k=0}^{\infty} b_{k}^{\C}  h^{k - \frac{n}{2} } \in S^{n/2,0}_{cla}; \,\,\, b_k^{\C} \in S^{0,0)}_{cla}, \,\, k=0,1,2,...,
\end{equation}
  where the $b_k^\C$'s denote the analytic continuation of the coefficients appearing in the formal solution
of the heat equation on $(X,g).$ In the following, to simplfiy notation,  we will simply write $b_k = b_k^{\C}; \, k=0,1,2,...$  for the symbols in the expansion (\ref{heat symbol}).
 \medskip

The Kahler potential
\begin{equation} \label{kahler potential}
2\rho(z) =  \Re r^2_{\C}(z, \Re z) = \frac{1}{4} r_{\C}^2(z, \bar{z}) =  |\xi|_{x}^2 \end{equation}
where, $z = \exp_{x} ( -i \xi).$

Using (\ref{kahler potential}) and the expansion in (\ref{hol heat}) it is proved in \cite [Theorem 0.1]{GLS}  that the operator $T_{hol}(h): C^{\infty}(
X) \to C^{\infty}(X_{\tau}^{\C})$ given by
\begin{equation} \label{holFBI}
T_{hol} \phi_h (z) = h^{-n/4} \int_{X} e^{ [ - r_{\C}^2(z,y)/2 - \rho(z)]/h} b^{\C}(z,y,h) \chi(x,y) \phi_h(y) dy, \quad z \in B_{\tau}^* \end{equation}
is also an FBI transform in the sense of \eqref{FBI} with $h$-elliptic amplitude  $b \in S^{n/2,0}_{cla}$ and phase function
$\phi(z,y) = i  \Big( \, \frac{ r_{\C}^2(z,y)}{2} + \rho(z) \, \Big). $
In \eqref{holFBI} the multiplicative factor $h^{-n/4}$ is added to ensure $L^2$-normalization so that  $\| T_{hol} \phi_h \|_{L^2(X_{\tau}^{\C})} \approx  1.$

Since $u_h$ are eigenfunctions of the Riemannian Laplacian on $(X,g)$ with eigenvalue $h^{-2}$ it follows by analyic continuation that 

$$ e^{-1/2h} u^{\C}(z) = E^{\C}(h) u_h (z); \quad z \in B_{\tau}^*M.$$

Consequently, in view of (\ref{holFBI}),

\begin{equation} \label{comp}
T_{hol}(h) u_h (z) = e^{-\rho(z)/h} E^{\C}(h) u_h (z) = e^{-1/2h} e^{-\rho(z)/h} u_h^{\C}(z). \end{equation} \

Using (\ref{hol heat}), it follows that the left parametrix $S_{hol}(h)$ in (\ref{left}) satisfies
$$  T^*_{hol}(h) |b_0|^{-2} T_{hol}(h) = S_{hol}(h) T_{hol}(h)  + O(h)_{L^2 \to L^2},$$
where $b_0 \in S^{0}_{cla}$ is $h$-elliptic principal symbol in (\ref{hol heat}).
Since $S(h) a T(h) = Op_{h}(a) + O(h^{\infty}),$ the semiclassical anti-Wick quantization, it follows by an application of standard 
$h$-pseudodifferential calculus that for any $a \in C^{\infty}_{0}(B^*_{\tau}),$
\begin{equation} \label{key psdo}
\langle a T_{hol} u_h, T_{hol} u_h \rangle_{L^2(B_{\tau}^*X)} = \langle Op_h(|b_0|^2 a) u_h, u_h \rangle_{L^2(X)} + O(h). \end{equation}  }

To simplify the writing somewhat, from now on we will fix the FBI transform and  set $T = T_{hol}.$

\subsubsection{ 2 microlocal quantum flux (2QF)}

The definition of the  second microlocal quantum flux (2QF) is in direct analogy with the usual quantum flux (QF) in section 2. Let $ \Sigma \subset B_{\tau}^*X$ be a hypersurface with normalized defining function $\beta$ and let $\chi \in C^{\infty}_0({\mathbb R},[0,1])$ be a cutoff as in Section 2. In the following, we continue to make the natural Bruhat-Whitney identification $B_{\tau}^*X \cong X^{\C}.$  As before, we let $u_h$ be  QE sequence of Laplace eigenfunctions, $a \in C^{\infty}_{0}(B^*_{\tau}X)$ and $\beta: B_{\tau}^*X \to {\mathbb R}$ be the normalized defining function for $\Sigma \subset B_\tau^*X$ with $\Sigma = \{ \beta = 0 \},$ $ \, d\beta |_{\Sigma} \neq 0$ and $\partial_{\beta}$ the unit exterior normal to $\Sigma.$ We define a modified {\em 2-microlocal quantum flux (2QF)} adapted to the conjugated operator

\begin{equation} \label{new p op}
P_{\rho}(h) = e^{-\rho/h} \cdot (-h^2 \Delta_{\bar{\partial}}) \cdot e^{\rho/h}, \end{equation}
where $\rho = \frac{1}{2} |\xi|_x^2$ is the Kahler potential  on the Grauert tube.

\begin{align} \label{2QF}
\langle & a hD_{\beta} \cdot T u_h, T u_h \rangle_{2QF(\Sigma^{\pm})}\nonumber \\
& := \lim_{\epsilon \to 0} \Big(  \frac{i}{h} \lll \, [ P_{\rho}(h),  a \, \chi (\beta/\epsilon) hD_{\beta} ]_{+} {\mathbbm{1}}_{\pm hD_{\beta} \geq 0} \, T u_h, T u_h \rrr_{B^*_{\tau}X)} \nonumber \\
 & \hspace{1cm} -  \frac{2 i}{h} \big\langle ( h \nabla \rho ) a  \chi(\beta/\epsilon)  hD_{\beta} \, {\mathbbm{1}}_{\pm hD_{\beta} \geq 0}   Tu_h, \, T u_h \big\rangle_{B_{\tau}^*X} \nonumber \\
 & \hspace{1cm} + \frac{2i}{h} \big\langle a \chi(\beta/\epsilon) hD_{\beta} \, {\mathbbm{1}}_{\pm hD_{\beta} \geq 0} Tu_h, \, (h \nabla \rho ) \, T u_h \big\rangle_{B_{\tau}^*X}   \Big).
\end{align}\\

We note that since $e^{-1/2h} e^{-\rho/h} u_h^{\C} = T(h) u_h,$ 
$$ P_{\rho}(h) T(h) u_h = e^{-1/2h} e^{-\rho/h} (-h^2 \Delta_{\bar{\partial}}) u_h^{\C} = 0.$$
%and so, (\ref{2QF}) is the direct Grauert tube analogue of the usual QF in (\ref{QEQF1}).  Similarly, the incoming (resp. outgoing) quantum flux norms are defined in analogy with .. by setting
%\begin{align} \label{2microQF}
% \langle a \cdot T u_h,  & T u_h \rangle_{2QF(\Sigma^{\pm})} \\ \nonumber
% & :=  \lim_{\epsilon \to 0} \frac{i}{h} \lll \, [ P_{\rho}(h),  a \, \chi (\beta/\epsilon) ]_{+}  {\mathbbm{1}}_{ \pm hD_{\beta} \geq 0} \, T_{hol} u_h, T_{hol} u_h \rrr_{L^2(B^*_{\tau}X)}.\end{align}

\subsection{Proof of Theorem \ref{2microlocalQER}} \begin{proof}   We  prove Theorem \ref{2microlocalQER} following the reasoning in Theorem \ref{theorem-1}. The result in Proposition \ref{2microlocal} will follow as a consequence. Let $\nu$ be a unit exterior normal of $\Sigma$ and suppose there is a domain $\Omega \subset B_{\tau}^* X$ with $\Sigma = \partial \Omega.$ In the following, we let $(\beta',\beta)$ be Fermi coordinates adapted to the hypersurface $\Sigma = \{ \beta = 0 \}$ with exterior unit normal $\partial_{\nu} = \partial_{\beta}.$ We note here that the $\beta$-coordinates are a priori, smooth functions in the canonical $(x,\xi)$-coordinates. 

By the same reasoning as in (\ref{key point}),

\begin{align} \label{2micro1}
\langle h D_{\nu} a T u_h, & T u_h \rangle_{2QF(\Sigma^{+})}  + \langle h D_{\nu} a T u_h, T u_h \rangle_{2QF(\Sigma^{-})}  \nonumber \\ \nonumber 
& = \lim_{\epsilon \to 0} \Big(  \frac{i}{h}   \big\langle [P_{\rho}(h), \chi(\beta/\epsilon)\psi(P_{\rho})] h
D_\nu a  T u_h, \, T u_h \big\rangle_{L^2(\Omega)}  \nonumber \\
& \quad -  \frac{2 i}{h} \big\langle ( h \nabla \rho ) a  \chi(\beta/\epsilon)  hD_{\beta} \, {\mathbbm{1}}_{\pm hD_{\beta} \geq 0}   Tu_h, \, T u_h \big\rangle_{B_{\tau}^*X} \nonumber \\
 & \quad + \frac{2i}{h} \big\langle a \chi(\beta/\epsilon) hD_{\beta} \, {\mathbbm{1}}_{\pm hD_{\beta} \geq 0} Tu_h, \, (h \nabla \rho ) \, T u_h \big\rangle_{B_{\tau}^*X}   \Big). \hspace{1in}
\end{align}

%= \frac{1}{(n-1)! \mu ( \{p_{\rho}=0\})} \int_\Sigma
%a(H_{p_\rho} \beta)|_{\Sigma} (\omega |_{\Sigma})^{n-1} + o(1),
%where in the last line of (\ref{2micro1}) we have applied QE using (\ref{key psdo})
%along with a simliar argument to the one in  Theorem 1.

Setting
$$A_{\epsilon}(h):= \chi(\beta/\epsilon)\psi(P_{\rho}) h
D_\beta \, a ,$$
where $a = a(\beta') \in S^0,$
it follows by an application of QE and taking $\epsilon \to 0$ limit that 
\begin{eqnarray} \label{rel1} \lim_{\epsilon \to 0} \frac{i}{h}   \big\langle [P_{\rho}(h), \chi(\beta/\epsilon)\psi(P_{\rho})] h
D_\nu a  T_{hol}u_h, \, T_{hol} u_h \big\rangle_{L^2(\Omega)}  \nonumber \\
= \lim_{\epsilon \to 0} \frac{i}{h}    \big\langle [P_{\rho}(h), A_{\epsilon}(h)] \,  Tu_h, \, Tu_h \big\rangle_{L^2(\Omega)} + o(1) \end{eqnarray}
as $h \to 0^+.$
We repeat the basic trick in the Rellich lemma using that $P_{\rho}(h) Tu_h = P_{\rho}(h) e^{-\rho/h} u_h^{\mathbb C} = 0$ to write

\begin{eqnarray} \label{rel2}
\frac{i}{h} \big\langle [P_{\rho}(h), A_{\epsilon}(h)] \,  Tu_h, \, T u_h \big\rangle_{L^2(\Omega)}  
= \frac{i}{h} \big\langle P_{\rho}(h) A_{\epsilon}(h) \,  Tu_h, \, Tu_h \big\rangle_{L^2(\Omega)} \nonumber \\
 -  \frac{i}{h} \big\langle A_{\epsilon}(h) \,  Tu_h, \, P_{\rho}(h) T u_h \big\rangle_{L^2(\Omega)}. \end{eqnarray}
 
 Unlike the standard Rellich case, one cannot apply Green's formula directly on the RHS of (\ref{rel2}) since $P_{\rho} = e^{-\rho/h} (-h^2\Delta_{\bar{\partial}}) e^{\rho/h}$ is not the Laplacian. To adapt the argument,  we write
 
 \begin{equation} \label{comm}
 P_\rho(h) = [e^{-\rho/h}, -h^2 \Delta_{\bar{\partial}}] e^{\rho/h}  - h^{2} \Delta_{\bar{\partial}} = -h^2 \Delta_{\bar{\partial}}  - 2 h \nabla \rho  - h \Delta_{\bar{\partial}}\rho \end{equation}
 Since by $L^2$-boundedness, $ \|  \frac{i}{h} [-h \Delta_{\bar{\partial}} \, \rho, A_{\epsilon} ]  \|_{L^2(\Omega)} = O(h),$ It follows that modulo $O(h)$-error the RHS in (\ref{rel2}) equals
 
 \begin{align} \label{rel3}
 \frac{i}{h} \big\langle -h^2 \Delta_{\bar{\partial}} A_{\epsilon}(h) \,  Tu_h, \, Tu_h \big\rangle_{L^2(\Omega)} 
 -  \frac{i}{h} \big\langle A_{\epsilon}(h) \,  Tu_h, \, -h^2 \Delta_{\bar{\partial}}Tu_h \big\rangle_{L^2(\Omega)}  \\ \nonumber \\ \nonumber 
- \frac{2 i}{h} \big\langle ( h \nabla \rho ) \,  A_{\epsilon}(h) \,  Tu_h, \, Tu_h \big\rangle_{L^2(\Omega)}
 + \frac{2i}{h} \big\langle A_{\epsilon}(h) \,  Tu_h, \, (h \nabla \rho ) \,  Tu_h \big\rangle_{L^2(\Omega)}.
 \end{align} \
   
 %\subsubsection{Estimating gradient terms in the Rellich expansion}
 To estimate the gradient terms appearing in the last line of (\ref{rel3}) we note that since they involve a difference of the form $ \frac{i}{h} ( \langle X Tu, Tu \rangle - \langle Tu, X Tu \rangle)$ where $X$ is an $h$-vector field rather than the Laplacian, these terms do not contribute boundary terms; rather they can added to the LHS in  (\ref{rel2}). So, after an application of Green's formula to rewrite the terms in the first line of (\ref{rel3}) as a boundary integral over $\Sigma = \partial \Omega,$ it follows that
 
 \begin{eqnarray} \label{rel3.5}
\frac{i}{h} \big\langle [P_{\rho}(h), A_{\epsilon}(h)] \,  Tu_h, \, T u_h \big\rangle_{\Omega} \hspace{1in} \nonumber \\
+ \frac{2 i}{h} \big\langle ( h \nabla \rho ) A_{\epsilon}(h) \,  Tu_h, \, T u_h \big\rangle_{\Omega} 
 - \frac{2i}{h} \big\langle A_{\epsilon}(h) \,  Tu_h, \, (h \nabla \rho ) \, T u_h \big\rangle_{\Omega}  \nonumber \\
 = i \Big( - \big\langle h \partial_{\beta} A_{\epsilon}(h) Tu_h, Tu_h \big\rangle_{\Sigma} + \big\langle A_{\epsilon}(h) Tu_h, h \partial_{\beta} Tu_h \big\rangle_{\Sigma} \Big) + O(h).
 \end{eqnarray} \

 %We expand the real $h$-vector field $h \nabla \rho$ into base and fiber components  and write $h\nabla \rho = \rho_{x}   \cdot \, h \partial_{x} +  \rho_{\xi} \cdot h \partial_{\xi}.$
 \subsubsection{Computation of the LHS in (\ref{rel3.5})} First, to compute the two terms in the penultimate line of (\ref{rel3.5}), we will need the follow lemma.  So as  not to  break the exposition here, we defer the proof of the following lemma  along with the computation of $q_1$ to the appendix.

 \begin{lemma} \label{sub}
 There exists $q_1 \in  C^{\infty}_0(\Omega) \subset S^{0}(\Omega)$ such that 
 \begin{eqnarray} \label{relgrad}
   \frac{2i}{h} \Big(  \big\langle  h \nabla \rho  \,  A_{\epsilon}(h) \,  Tu_h, \, T u_h \big\rangle_{L^2}
 - \big\langle A_{\epsilon}(h) \,  Tu_h, \,  h \nabla \rho \, T u_h \big\rangle_{L^2} \Big) \nonumber \\
 = \langle Op_h(q_1) u_h, u_h \rangle_{L^2} + O(h).
 \end{eqnarray}

 \end{lemma}

 As for the first commutator term in (\ref{rel3.5}), the principal symbol of the commutator in Fermi coordinates  $(\beta',\beta): U_{\Sigma} \to {\mathbb R}^{2n}$ in a tube centered on $\Sigma = \{ \beta = 0 \}$ (see subsection \ref{FERMI}). We denote the symplectic coordinates dual to $(\beta',\beta)$ by $(\xi_{\beta'}, \xi_{\beta})$ and split the gradient of $\rho$ into tangential and normal parts by writing $\nabla \rho = \rho_{\beta} \partial_{\beta} + \rho_{\beta'} \cdot \partial_{\beta'}.$  
 \begin{eqnarray} \label{comm1}
  \sigma \big( \frac{i}{h} [ P_{\rho}, A_{\epsilon} ] \big) &=& \Big\{ \, \xi_{\beta}^2 + R(\beta,\beta',\xi_{\beta'}) - 2 i  ( \rho_{\beta} \xi_{\beta} + \rho_{\beta'} \cdot \xi_{\beta'} ), \, \chi'(\beta/\epsilon) \, a \,\psi(p_{\rho}) \xi_{\beta} \Big\} \nonumber \\
  &=& 2 a  \xi_{\beta}^2  \epsilon^{-1}  \chi'(\beta/\epsilon)  -2 i  a \rho_{\beta}  \xi_\beta  \epsilon^{-1} \chi'(\beta/\xi) 
  + d (\beta, \beta', \xi_\beta, \xi_{\beta'})  \chi(\beta/\epsilon),
\end{eqnarray}
 where $d \in S^{0}(U_{\Sigma})$ From (\ref{comm1}). Then, by $L^2$-boundedness,  
 \begin{eqnarray} \label{comm2}
 \frac{i}{h} \big\langle [P_{\rho}(h), A_{\epsilon}(h)] \,  Tu_h, \, Tu_h \big\rangle_{\Omega} &=&   2 \big\langle \epsilon^{-1} \chi'(\beta/\epsilon) \, a \, (h D_{\beta})^2 \,  T u_h, T u_h \big\rangle_{\Omega}  \nonumber\\
 &-& 2 i \big\langle \epsilon^{-1} \chi'(\beta/\epsilon) \,a \, \rho_{\beta}\,  h D_{\beta} \,Tu_h, Tu_h \big\rangle_{\Omega} \nonumber \\ 
 &+& O(1) \| \chi(\beta/\epsilon) T u_h \|^2_{L^2}. \end{eqnarray}
 
 It follows by an  application of QE that
 $$ \| \chi(\beta/\epsilon) T u_h \|_{L^2}^2  = O(1) \int_{S^*X}  \chi^2(\beta/\epsilon) d\sigma + o_{\epsilon}(1) = O(\epsilon) + o_{\epsilon}(1)$$
 and so, from (\ref{comm2})
 
  \begin{eqnarray} \label{comm3}
 \frac{i}{h} \big\langle [P_{\rho}(h), A_{\epsilon}(h)] \,  Tu_h, \, Tu_h \big\rangle_{\Omega} = 2 \big\langle  \epsilon^{-1} \chi'(\beta/\epsilon) \,a (h D_{\beta})^2 T u_h, T u_h \big\rangle_{\Omega} \nonumber \\
   - 2 i \big\langle \epsilon^{-1} \chi'(\beta/\epsilon) \, a \, \rho_{\beta}  h D_{\beta} Tu_h, Tu_h \big\rangle_{\Omega} + O(\epsilon) + o_{\epsilon}(1).
 \end{eqnarray}

%\cs check this \cs Since $ \nabla \rho = \sum_{j=1}^{n} \xi_j \partial_{\xi_j}$ and $\langle \partial_{x_j}, \partial_{\xi_k} \rangle_h = 0$ it follows that $\rho_{\beta} h \partial_{\beta} = \sum_{j=1}^{n} \rho_j(x,\xi) h \partial_{\xi_j}$ where $\rho_j \in C^{\infty}(U_\Sigma).$ Then, from (\ref{intertwine2}) it follows that
% $$ \big\langle \epsilon^{-1} \chi'(\beta/\epsilon) \, a \, \rho_{\beta}  h D_{\beta} Tu_h, Tu_h \big\rangle_{\Omega} = O_{\epsilon}(h).$$

 By direct computation using the intertwining relations in (\ref{intertwine1}), (\ref{intertwine2}), and (\ref{intertwine3}) (see appendix),

 \begin{eqnarray} \label{comm4}
  2 \big\langle  \epsilon^{-1} \chi'(\beta/\epsilon) \,a (h D_{\beta})^2 T u_h, T u_h \big\rangle_{\Omega}  - 2 i \big\langle \epsilon^{-1} \chi'(\beta/\epsilon) \, a \, \rho_{\beta}  h D_{\beta} Tu_h, Tu_h \big\rangle_{\Omega}  \nonumber \\
  = \big\langle Op_h ( \eta_{\epsilon}) u_h, u_h \big\rangle_{L^2} + O(h), \nonumber \\ \nonumber \\
  \text{where} \quad \eta_{\epsilon}(x,\xi) = a \epsilon^{-1} \chi'(\beta/\epsilon)  |b_0|^2 \Big(  ( \xi \cdot \partial_{\beta}x   \big)^2  - \rho_\beta  ( \xi \cdot \partial_{\beta} x )\Big),
  \end{eqnarray}
  and in (\ref{comm4})  we use the shortform notation $ \xi \cdot  \partial_{\beta} x = \sum_{j=1}^{n}  \xi_j \cdot ( \partial_{\beta} x_j).$
  From (\ref{comm3}) and (\ref{comm4}) it follows by an application of QE that  
  
  \begin{eqnarray} \label{LHSupshot1}
 \frac{i}{h} \big\langle [P_{\rho}(h), A_{\epsilon}(h)] \,  Tu_h, \, Tu_h \big\rangle_{\Omega} \nonumber \\
 = \int_{S^*X} \epsilon^{-1} \chi'(\beta/\epsilon) \, q_2(x,\xi) \, d\sigma + O(\epsilon) + o_{\epsilon}(1), \nonumber \\
\end{eqnarray} 
where 
\begin{equation} \label{q2term}
 q_2(x,\xi) = a \,   |b_0(x,\xi,x)|^2  \Big( \big( \xi \cdot \partial_{\beta} x \big)^2 \, - \rho_\beta (\xi  \cdot \partial_{\beta} x ) \Big). \end{equation}\

 \subsubsection{Computation of the RHS in (\ref{rel3.5})}

 Finally, to compute the RHS in (\ref{rel3.5}) we use that in Fermi coordinates $(\beta',\beta)$ adapted to $\Sigma = \{ \beta = 0 \}$ the Kahler Laplacian is of the form
 $-h^2 \Delta_{\bar{\partial}} = -(h \partial_{\beta})^2 + R(\beta,\beta',hD_{\beta'}).$  Here $R(0,\beta',hD_{\beta'})= - h^2 \Delta_{\Sigma},$ the induced tangential Laplacian along $\Sigma.$

  Since the operator
 $$P_{\rho}(h) = - (h \partial_\beta)^2 + R(\beta,\beta', h D_{\beta'}) - 2 h \nabla \rho(\beta,\beta') - h (\Delta \rho) $$
 and $P_{\rho}(h) T u_h = 0,$ it follows that
 
 \begin{align} \label{1termRHS}
-i \big\langle h \partial_{\beta} & A_{\epsilon}(h) Tu_h,  Tu_h \big\rangle_{\Sigma} \nonumber \\
 & = -  \big\langle h \partial_{\beta} ( \chi(\beta/\epsilon) a  h \partial_{\beta}) Tu_h, Tu_h \big\rangle_{\Sigma} \nonumber \\
& = -  \big\langle a \big( -h^2 \Delta_{\Sigma} - 2h \nabla \rho - h \Delta \rho \big)  Tu_h, Tu_h \big\rangle_{\Sigma}    - h  \big\langle (\partial_{\beta}a ) \,  h\partial_{\beta} Tu_h, Tu_h \big\rangle_{\Sigma} \nonumber \\
 & =  \big\langle a \big( h^2 \Delta_{\Sigma} + 2h \nabla \rho + h \Delta \rho \big)  Tu_h, Tu_h \big\rangle_{\Sigma}.   \end{align}

 In the last line of (\ref{1termRHS}) we have used that $R(0,\beta',hD_{\beta'}) = -h^2 \Delta_{\Sigma}$ where the latter is the induced $h$-Laplacian along $\Sigma$ and $\partial_{\beta} a(\beta') = 0.$ We have also used that $\chi(0) = 1$ and $\chi'(0) = 0.$ 
 %We note in passing that if the symbol $a = a(\beta') \chi(\beta)$ (i.e. it is tangential ), then the last term in (\ref{1termRHS}) vanishes since $\partial_{\beta} a |_{\Sigma} = 0$ in that case. In any event, by Sobolev restriction,
% $h  \big\langle (\partial_{\beta}a ) \,  h\partial_{\beta} Tu_h, Tu_h \big\rangle_{\Sigma} = O(h^{1/2 -0})$ and so the last term in (\ref{1termRHS}) can be absorbed in the $o_{\epsilon}(1)$- error term.

 As  for the second boundary term, since $\chi(0) = 1,$
 
 \begin{eqnarray} \label{2termRHS}
i \big\langle A_{\epsilon}(h) Tu_h, h \partial_{\beta} Tu_h \big\rangle_{\Sigma} 
&=& \big\langle \chi(\beta/\epsilon) a  h \partial_{\beta}) Tu_h,  h \partial_{\beta} Tu_h \big\rangle_{\Sigma} \nonumber \\
 &=& \big\langle a  \,  h \partial_{\beta} Tu_h, \, h \partial_{\beta} Tu_h \big\rangle_{\Sigma}. 
 \end{eqnarray}\
 
\begin{remark} We note here that  Proposition \ref{2microlocal} follows from (\ref{rel3.5}) since in view of (\ref{1termRHS}) and (\ref{2termRHS}), the boundary terms appearing on the RHS are independent of $\epsilon$, so one can take $\lim_{\epsilon \to 0}.$
\end{remark}

From Lemma \ref{sub}, (\ref{LHSupshot1}), (\ref{1termRHS}) and (\ref{2termRHS}), it follows that

\begin{eqnarray} \label{UPSHOT}
\int_{S^*X}  \epsilon^{-1} \chi'(\beta/\epsilon) \, a\, ( q_1 + q_2) \, d\sigma + O(\epsilon) + o_{\epsilon}(1) \hspace{1in}\nonumber \\ \nonumber \\
= \big\langle  \big( h^2 \Delta_{\Sigma} + 2h \nabla \rho + h \Delta \rho \big)  Tu_h, Tu_h \big\rangle_{\Sigma}    
+ \big\langle a  \,  h \partial_{\beta} Tu_h, \, h \partial_{\beta} Tu_h \big\rangle_{\Sigma}. 
%=  \big\langle  - ah^2 \Delta_{\Sigma}   Tu_h, Tu_h \big\rangle_{\Sigma}    
%+ \big\langle a  \,  h \partial_{\beta} Tu_h, \, h \partial_{\beta} Tu_h \big\rangle_{\Sigma}  + O(h^{1/2-0}).
\end{eqnarray}\\

%The last line in (\ref{UPSHOT}) follows since $\nabla \rho = \nabla_{\tilde{g}} \rho = \xi \cdot \partial_{\xi},$ it follows from 
% (\ref{intertwine2}) that $ \langle  a h \nabla \rho T u_h, Tu_h \rangle_{\Omega} = O(h)$ and so, it follows by Sobolev restriction that   $ \langle  a h \nabla \rho T u_h, Tu_h \rangle_{\tilde{\Sigma}} = O(h^{1/2-0}).$ 
 
 To complete the proof of  take $h \to 0^+$ in (\ref{UPSHOT}) and then  $\epsilon \to 0^+.$ We note that the RHS in (\ref{UPSHOT}) which is the sum of the 2-microlocal QF terms on the LHS of Proposition \ref{2microlocal} is actually independent of $\epsilon>0.$

 Finally, writing $e^{-1/2h} e^{-\rho/h} u_{h}^{\C} = T_{hol}u_h$ in (\ref{UPSHOT}) completes the proof of Theorem 3 in the case of general separating hypersurfaces $\Sigma \subset X^{\mathbb C}.$

\end{proof}

\section{Appendix: Proof of Lemma \ref{sub}}

\begin{proof} Since the expression in Lemma \ref{sub} is not a commutator, we must compute the leading order terms in each term separately. 
First, we compute the asymptotics of $ I_1(h):= \frac{2i}{h} \langle h \nabla \rho A_\epsilon Tu_h, Tu_h \rangle_{L^2}.$
By direct computation, using  (see (\ref{grad})) $\nabla \rho = \nabla_{\tilde{g}}\rho = \xi \cdot \partial_{\xi},$
$$ h\nabla \rho A_\epsilon Tu(x,\xi) = (2\pi h)^{-3n/4} \int_X ( \xi \cdot h \partial_{\xi} ) ( a \chi(\beta/\epsilon) hD_{\beta}) ( e^{i\phi/h} b ) u(y) dy$$
$$ = (2\pi h)^{-3n/4} \int_X ( \xi \cdot h \partial_{\xi} )  \Big(  e^{i\phi/h}  [ a \chi(\beta/\epsilon) ( \partial_{\beta} \phi)  b  +  a \chi(\beta/\epsilon) h D_{\beta} b ] \Big) u(y) dy$$
$$= (2\pi h)^{-3n/4}\int_X e^{i\phi/h} \big(  i \xi \cdot (\partial_{\xi} \phi) c   + h \xi \cdot \partial_{\xi} c  \big) u(y) dy$$
where $c = a \chi(\beta/\epsilon) \Big( ( \partial_{\beta} \phi)  b  + h D_{\beta} b \Big).$
It follows that $I_1(h)$ equals
%\begin{align} \label{I1}
%  (2 \pi h)^{-3n/2} (2 i h^{-1}) \int e^{i \phi(x,\xi,y) - i \overline{\phi(x,\xi,z)}/h} \, \big( i \xi \cdot (\partial_{\xi} \phi(x,y,\xi)) c(x,y,\xi)   + h \xi \cdot \partial_{\xi} c(x,y,\xi)  \big)  \nonumber \\
% \times   \overline{ b(x,\xi,z)} u(y) u(z) dy dz d\xi dx.
%\end{align}
\begin{align} \label{I1}
  (2 \pi h)^{-3n/2} (2 i h^{-1}) \int e^{i \phi(x,\xi,y) - i \overline{\phi(x,\xi,z)}/h} A_{cb}(x,y,\xi,z) u(y) u(z) dy dz d\xi dx.
\end{align}
where
\[
A_{cb} = \big( i \xi \cdot (\partial_{\xi} \phi(x,y,\xi)) c(x,y,\xi)   + h \xi \cdot \partial_{\xi} c(x,y,\xi)  \big)     \overline{ b(x,\xi,z)}.
\]
One applies stationary phase in $x$ in (\ref{I1}) with critical point $x_c = \frac{y+z}{2} (   1 + O(y-z) ).$ Since $\Re \phi(x_c,\xi, y) = \frac{1}{2} \langle z-y,\xi \rangle$ and 
$\Im [ \phi(x_c, \xi,y) - \overline{\phi(x_c,\xi,z)} ] = |y-z|^2 + O(|y-z|^3),$ it follows that (\ref{I1}) equals

\begin{align} \label{I1.1}
& (2\pi h)^{-n}  (2ih^{-1}) \int e^{i\langle z-y, \xi \rangle/2h}    \big( i \xi \cdot (\partial_{\xi} \phi(x_c,y,\xi)) c(x_c,y,\xi)   + h \xi \cdot \partial_{\xi} c(x_c,y,\xi)  \big) \nonumber \\ 
 & \quad  \quad \times   \overline{ b_0(x_c,\xi,z)} e^{[ -|y-z|^2  + O(y-z)^3  ]/h}   u(y) u(z) dy dz d\xi  \nonumber \\ 
  &  \quad + (2\pi h)^{-n} (2ih^{-1}) \int e^{i\langle z-y, \xi \rangle/2h} \nonumber \\
  & \quad \quad \times  h \Delta_x \Big(  [  i \xi \cdot (\partial_{\xi} \phi(x,y,\xi)) \, c(x,y,\xi)   + h \xi \cdot \partial_{\xi} c(x,y,\xi)  ]  \nonumber \\
 &  \quad \quad \times   \overline{ b_0(x,\xi,z)}  \Big) |_{x = x_c} \, e^{[ -|y-z|^2  + O(y-z)^3  ]/h}  u(y) u(z) dy dz d\xi  + O_{L^2 \to L^2}(h). 
\end{align}

Since  $\partial_{\xi} \phi(x_c,y,\xi)) = \frac{ (z-y)}{2} + O(|y-z|^2)$ and by Taylor expansion, $$e^{[ -|y-z|^2  + O(y-z)^3  ]/h}  = 1 + O( [|y-z|^2  + O(y-z)^3  ]/h ),$$ it follows by integration by parts in (\ref{I1.1}) with respect to $h D_{\xi}$ that the first integral in (\ref{I1.1}) can be written in the form

\begin{align} \label{I1.2}
 (&2\pi h)^{-n} \int e^{i\langle z-y, \xi \rangle/2h}  ( -2 i \partial_{\xi} ) \cdot  \Big( \xi \,  c(x_c,y,\xi)   \, \overline{b_0(x_c,\xi,z)}  \Big) \, u(y) u(z) dy dz d\xi   \hspace{2in} \nonumber \\
  &+  (2\pi h)^{-n} \int e^{i\langle z-y, \xi \rangle/2h} \, \Big(  2 i \xi \cdot \partial_{\xi} c(x_c,y,\xi)  
 \times \overline{b_0(x_x,\xi,z)} \Big)    u(y) u(z) dy dz d\xi  + O(h). \hspace{2in}\nonumber \\
  \end{align}
  It follows by application of QE that the leading terms in (\ref{I1.2}) arise when $\partial_\xi$ hits the cutoff $\chi(\beta/\epsilon)$ in $c(x_c,\xi,y)$; all others terms are $O(\epsilon) + o_{\epsilon}(1).$ These terms appear with opposite signs in the two summands  and so, the contribution from (\ref{I1.2}) is $O(\epsilon) + o_{\epsilon}(1).$ \

%  The resulting symbol in the amplitude is $i \xi \cdot \partial_{\xi} ( \chi(\beta/\epsilon)) (\partial_{\beta} \phi) a |b|^2 |_{x=x_c}.$
  
  As for the second integral  in (\ref{I1.1}), since $\partial_{\xi} \phi(x-y,\xi) = O(x-y),$ again by integration by parts with respect to $hD_{\xi},$ it equals
  \begin{eqnarray} \label{I1.3}
  (2\pi h)^{-n} \int e^{i\langle z-y, \xi \rangle/2h}   
   \Big(  [  -4 \xi \cdot ( \partial_x \partial_{\xi}  \phi(x,y,\xi)) \, \partial_x c(x,y,\xi)   ]  \,\,
    \overline{ b_0(x,\xi,z)}  \Big) |_{x = x_c} \, \hspace{1in}\nonumber \\
    \times u(y) u(z) dy dz d\xi  + O_{L^2 \to L^2}(h) \hspace{1in} \nonumber \\
  = (2\pi h)^{-n} \int e^{i\langle z-y, \xi \rangle/2h}  [ - 4 a (\partial_{\beta} \phi) |b_0|^2 \xi \cdot   \partial_x  ( \chi(\beta/\epsilon) ) |_{x = x_c}  u(y) u(z) dy dz d\xi \hspace{1in} \nonumber \\
    + O(\epsilon) + o_{\epsilon}(1) \hspace{3.5in} \nonumber \\
    = (2\pi h)^{-n} \int e^{i\langle z-y, \xi \rangle/2h}   \, \Big(  - 4 a (\partial_{\beta} \phi) |b_0|^2  \, \big( \xi \cdot  \partial_{x} \beta \big) \,  \epsilon^{-1}  \chi'(\beta/\epsilon)     \Big) |_{x = x_c}  u(y) u(z) dy dz d\xi \hspace{1in} \nonumber \\
    + O(\epsilon) + o_{\epsilon}(1) \hspace{3.5in}
\end{eqnarray}
  
  In (\ref{I1.3}) we have used that in the Leibniz expansion for the Laplacian $\Delta_x$, to get the leading contribution, one derivative $\partial_{x_j}$ must hit $\partial_{\xi} \phi$ and another $\partial_{x_j}$ must hit the $\chi(\beta/\epsilon)$-term  in $c$; all other terms are $O(\epsilon) + o_{\epsilon}(1)$ by QE.\\

 % \cs change $\xi$ to $-2 \xi$ to get psdo \cs
  
 % In view of (\ref{I1.2}) and (\ref{I1.3}), it follows that
 % $$ I_1(h) = \int_{S^*X} q_1^{(1)} d\sigma + O(\epsilon) + o_{\epsilon}(1),$$
 % where
 % \begin{eqnarray} \label{eta11}
 %  q_1^{(1)} = \Big( -4a (\partial_{\beta} \phi) |b|^2 \xi \cdot   \partial_x  ( \chi(\beta/\epsilon) )  \Big) |_{x = x_c}.
 % \end{eqnarray}\\

For the second term $I_2(h):=  -2i h^{-1} \big\langle A_{\epsilon}(h) \,  Tu_h, \,  h \nabla \rho \, T u_h \big\rangle_{L^2}$ in Lemma \ref{2microlocal}, a similar argument to the one above for $I_1(h)$ gives

\begin{eqnarray}\label{I2.1}
A_{\epsilon} Tu(x,\xi) = (2\pi h)^{-3n/4} \int a \chi (\beta/\epsilon) hD_{\beta} (e^{i\phi/h} b ) u(y) dy \nonumber \\
= \int e^{i \phi(x,\xi,y)/h} a(x,\xi) \chi(\beta/\epsilon) ( \partial_\beta \phi) b_0 u(y) dy + O_{L^2}(h).
\end{eqnarray}

Similarily, using that $\nabla \rho = \xi \cdot \partial_{\xi},$
\begin{eqnarray} \label{I2.2}
 h \nabla \rho Tu_h(x,\xi) = (2\pi h)^{-3n/4} \int  e^{i \phi/h}  \Big(  i \xi \cdot  (\partial_{\xi} \phi(x,\xi,y) )  b  \nonumber \\
 + h \xi \cdot (\partial_{\xi} b) \Big) u(y) dy 
\end{eqnarray}

So, from (\ref{I2.1}) and (\ref{I2.2}), $I_2(h)$ equals
\begin{eqnarray} \label{I2.3}
-(2ih^{-1}) (2\pi h)^{-3n/2}  \int e^{i \phi(x,\xi,y) - i \overline{\phi(x,\xi,z)}/h} a(x,\xi) \chi(\beta/\epsilon) ( \partial_\beta \phi) b_0(x,\xi)    \nonumber \\
  \times   \overline{ \big(  i \xi \cdot  (\partial_{\xi} \phi )  b_0  + h \xi \cdot (\partial_{\xi} b_0)\big) } u(y) u(z) dy dz d\xi dx + O(h). 
\end{eqnarray}

Again, one applies stationary phase in $x$ in (\ref{I2.3}). Since now the first term does not involve any differentiation of $\chi(\beta/\epsilon)$, after applying QE, that term is $O(\epsilon) + o_{\epsilon}(1).$ The next term in the stationary phase expansion involves application of $h \Delta_x$ to the amplitude in (\ref{I2.3}). Here, to isolate the leading order term,  by Liebniz rule one $x$-derivative  hits $\chi(\beta/\epsilon)$ and another one hits $\partial_{\xi} \phi(x,\xi,y)$ with $\partial_x \partial_{\xi} \phi = 1.$ Here, as above, we note that $\partial_{\xi} \phi(x_c,\xi,y) = \frac{y-z}{2}$ and so, if an  $x$-derivative does not hit this term, by integration by parts with respect to $hD_{\xi}$ in (\ref{I2.3}) one gets an $O(h)$ contribution. The end result is that

\begin{eqnarray} \label{I2.4}
I_2(h) =  I_1(h) + O(\epsilon) + o_{\epsilon}(1),
\end{eqnarray}
where $I_1(h)$ is given (\ref{I1.3}).

Finally, we note that to change (\ref{I1.3}) into an $h$-psdo, we rescale the frequency variables $\xi \mapsto -2 \xi.$ The result is that
\begin{eqnarray} \label{I2.5}
I_1(h) + I_2(h) = \langle Op_h(q_1) u, u \rangle + O(\epsilon) + o_{\epsilon}(1),
\end{eqnarray} 
where setting $\eta = -\xi /2,$

\begin{eqnarray} \label{q1term}
q_1 (x,\eta) =   8 a(x,\xi)  |b_0(x,\xi,x)|^2 (\partial_{\beta} \phi)(x,\xi)  \nonumber \\
\times  \big( \eta \cdot \partial_{x} \beta(x,\xi) \big) \frac{1}{\epsilon} \chi'(\beta/\epsilon) |_{\xi = -2\eta}.
\end{eqnarray}

%\begin{eqnarray}
%I_2(h) = \int_{S^*X} \frac{1}{\epsilon} \chi'(\beta/\epsilon) q_1^{(2)}) d\sigma + O(\epsilon) + o_{\epsilon}(1) \nonumber \\
%= \int_{S^*X \cap \Sigma} q_1^{(2)}  d\sigma_{\Sigma} + O(\epsilon) + o_{\epsilon}(1),
%\end{eqnarray}

%where   

%\begin{eqnarray} \label{etaI2}
%q_1^{(2)} =   2 a |b|^2 (\partial_{\beta} \phi) \xi \cdot \partial_{x} \beta \frac{1}{\epsilon} \chi'(\beta/\epsilon).
%\end{eqnarray}

%In view of (\ref{eta11}) and (\ref{etaI2}), the lemma follows with  $q_1 = q_1^{(1)} + q_1^{(2)}.$
\end{proof}

%By direct computation (see appendix),
 %$$ (h D_{\beta})^2 T u_h (x,\xi) = \int_{X} (h D_{\beta})^2 ( e^{i\phi(x,\xi,y)/h} b(x,\xi,y) ) u_h(y) dy  $$
% $$= \int_{X} e^{i\phi(x,\xi,y)/h} \big(    ( \partial_{\beta} \phi )^2  + 2 h (\partial_{\beta} \phi) \cdot ( D_{\beta} b ) + h^2 D_{\beta}^2 b \big) u(y) dy$$
% $$= \int_{X} e^{i\phi(x,\xi,y)/h}   ( \partial_{\beta} \phi(x,\xi,y)  )^2  b(x,\xi,y) u(y) dy + O_{L^2}(h).$$
 
% It follows that
 %\begin{eqnarray} \label{comm5}
 % \langle \epsilon^{-1} \chi(\beta/\epsilon) (h D_{\beta})^2 Tu_h, Tu_h \rangle  = \int e^{i \phi(x,\xi,y) - i \overline{\phi(x,\xi,z)}/h} a(\beta,\beta') \epsilon^{-1} \chi(\beta/\epsilon) ( \partial_{\beta} \phi)^2(x,\xi,y)  \nonumber \\
%  \times  b(x,\xi,y) \overline{ b(x,\xi,z)} u(y) u(z) dy dz d\xi dx.
  %\end{eqnarray}

 % \begin{eqnarray} \label{comm1}
%  \Big\langle \frac{h}{i} [ P_{\rho}, A_{\epsilon} ] T u_h, Tu_h  \Big\rangle_{\Omega} &=&   \int_{S^*X}  \big( 2 a  \xi_{\beta}^2   - i  a \rho_{\beta}  \xi_\beta \big)  \epsilon^{-1} \chi(\beta/\xi)  \, d\sigma \nonumber \\
 % & +& \int_{S^*X}  d (\beta, \beta', \xi_\beta, \xi_{\beta'}) \ \chi(\beta/\epsilon) \, d\sigma  + o_\epsilon(1).
%  \end{eqnarray}\

\bibliographystyle{alpha}
\bibliography{HC-bib}

\end{document}